\documentclass[11pt]{article}
\usepackage{amsmath,amsthm,amsfonts,amssymb,amscd, amsxtra, mathrsfs,commath,url}

\numberwithin{equation}{section}

\usepackage{amstext,amssymb}

\usepackage[margin=2.7 cm,nohead]{geometry}
\usepackage{bbm}

\usepackage[table]{xcolor}
\usepackage{float}
\usepackage{threeparttablex,color,soul,colortbl,hhline,rotating,booktabs,longtable,multirow,makecell}

\usepackage[caption=false]{subfig} 
\usepackage{algorithmic, enumitem}
\newtheorem{theorem}{Theorem}
\newtheorem{lemma}[theorem]{Lemma}
\newtheorem{definition}{Definition}
\newtheorem{corollary}[theorem]{Corollary}
\newtheorem{proposition}[theorem]{Proposition}

\newtheorem{remark}{Remark}

\newtheorem{algorithm}{Algorithm}

\def\diam{\operatorname{diam}}

\def\argmin{\operatorname{argmin}}

\def\min{\displaystyle\operatorname{min}}
\def\max{\operatorname{max}}

\newcommand{\ds}{\displaystyle}
\newcommand{\R}{\mathbb{R}}

\newcommand{\J}{{\cal J}}
\newcommand{\K}{\mathbb{K}}
\newcommand{\N}{\mathbb{N}}
\newcommand{\subsetinf}{\displaystyle\mathop{\subset}_{\infty}}

\definecolor{lightgray}{gray}{0.95}

\newcommand{\myscale}{0.45}
\newcommand{\myscaletwo}{0.40}

\newtheorem*{Armijo}{Armijo step size}
\newtheorem*{Adaptive}{Adaptive step size}
\newtheorem*{diminishing}{Diminishing step size}
 
\begin{document}
\title{A generalized conditional gradient method  for  multiobjective  composite optimization problems}
\author{
P. B.  Assun\c c\~ao\thanks{Instituto de Matem\'atica e Estat\'istica, Universidade Federal de Goi\'as,  CEP 74001-970 - Goi\^ania, GO, Brazil, E-mails: {\tt  pedro.filho@ifg.edu.br},  {\tt  orizon@ufg.br},  {\tt  lfprudente@ufg.br}. The authors was supported in part by  CNPq grants 305158/2014-7 and 302473/2017-3, CAPES.}
\and
O.  P. Ferreira\footnotemark[1]
\and
L. F. Prudente\footnotemark[1]
}
\maketitle
\noindent
{\bf Abstract:} 
This article deals with multiobjective composite optimization problems that consist of simultaneously minimizing several objective functions,  each of which is composed of a combination of smooth and non-smooth functions. To tackle these problems, we propose a generalized version of the  conditional gradient method, also known as Frank-Wolfe method. The method is analyzed with  three step size strategies, including Armijo-type, adaptive,  and diminishing step sizes.  We establish asymptotic convergence properties and iteration-complexity bounds, with and without convexity assumptions on the objective functions.  Numerical experiments illustrating the practical behavior of the methods are presented.

\vspace{12pt}
\section{Introduction}
Multiobjective optimization problems typically involve the simultaneous minimization of multiple and conflicting objectives.
A solution to the problem leads to a set of alternatives with different trade-offs between the objectives. 
In fact, in this scenario, we use the concept of {\it Pareto optimality} to caracterize a solution.
In summary, a point is called {\it Pareto optimal} if, with respect to this point, none of the objective functions can be improved without degrading another.
One strategy for computing Pareto points that has become very popular consists of extending methods for scalar-valued optimization to vector-value optimization, rather than using scalarization approaches~\cite{geoffrion1968proper}.
To the best of our knowledge, this strategy was coined in the work \cite{FliegeSvaiter2000} that proposed the steepest descent methods for unconstrained multiobjective optimization. 
Since of then, new properties related to this method  have been discovered and several variants of it  have been considered, see for example  \cite{BelloCruz2013,  FliegeVazVicente2018, FukudaDrummond2011, FukudaDrummond2013,  Drummond2004, Drummond2005,ellen}.  
In recent years,  there has been a significant increase in the number of papers  addressing  concepts, techniques,  and methods for  multiobjective  optimization, see for example \cite{ansary,Bento2018, Carrizo2016,doi:10.1137/10079731X,Fliege.etall2009,Fliege2016,Goncalves2022,PerezPrudente2018, Prudente2022,Tabatabaei2018,ThomannEichfelder2019,wang2019extended}.

In the present paper, we consider {\it multiobjective composite optimization problems}, where the objective function  $F:\mathbb{R}^n \to (\R\cup\{+\infty\})^m$, given by $F(x):=(f_1(x),\ldots,f_m(x))$,  has the following special separable structure:
$$f_j(x):= g_j(x)+h_j(x), \quad \forall j=1,\ldots,m,$$
where, for each $j=1,\ldots,m$, $g_j:\R^n\to \R\cup\{+\infty\}$ is proper, convex, and lower semicontinuous, and $h_j:\R^n\to \R$  is continuously differentiable.
We define $G:\mathbb{R}^n \to(\R\cup\{+\infty\})^m$ by $G(x):=(g_1(x),\ldots,g_m(x))$ and $H:\mathbb{R}^n \to \mathbb{R}^m$ by $H(x):=(h_1(x),\ldots,h_m(x))$, and denote this problem by
\begin{equation}\label{eq:Problem1} 
\min_{x\in {\mathbb{R}^n}}{F(x):=G(x)+H(x)}.
\end{equation}
An important instance of \eqref{eq:Problem1}  is obtained when $G$ is the indicator function (in a vector sense) of a given set  ${\cal C}\subset \R^{n}$, i.e.,  for all $j=1,\ldots,m$, $g_j(x)=0$ for all $ x \in {\cal{C}}$ and $g_j(x) =+\infty$ otherwise.
In this case,  \eqref{eq:Problem1} merges into  the following  constrained multiobjective  optimization problem 
 \begin{equation} \label{eq:ProblemV2} 
\min_{x\in{\cal{C}}}{H(x)}.
\end{equation}
Furthermore, as discussed in \cite{TanabeFukudaYamashita2019}, the separable structure in \eqref{eq:Problem1}  can be used to model robust multiobjective optimization problems, which are problems that include uncertain parameters and the optimization process is considered under the worst scenario.

As far as we know, \cite{doi:10.1080/02331934.2018.1440553} was the first paper to deal with problem~\eqref{eq:Problem1}, where a forward--backward proximal point type algorithm was studied.
In \cite{TanabeFukudaYamashita2019}, a proximal gradient method to solve problem~\eqref{eq:Problem1} was proposed, see also \cite{doi:10.1080/02331934.2020.1800699,doi:10.1007/s11590-022-01877-7}.
More recently,  some Newton-type approaches were considered in \cite{doi:10.1080/10556788.2022.2157000,https://doi.org/10.48550/arxiv.2108.00125}.
It is worth mentioning that  a version of the  {\it conditional gradient method} also known as {\it Frank-Wolfe algorithm}, see  \cite{FrankWolfe1956,  LevitinPolyak1966},  to solve  \eqref{eq:ProblemV2}  was proposed and analyzed in \cite{Assunccao2021}. However,  a generalized version of this method to solve \eqref{eq:Problem1}  has not yet been considered.

In this paper,  a  multiobjective  version   of the scalar  generalized conditional gradient method~\cite{Beck2014, BrediesLorenzMaass2009, RakotomamonjyFlamaryCourty2015}  to solve problem~\eqref{eq:Problem1} is proposed. 
The method is analyzed with  three step size strategies, including Armijo-type, adaptive,  and diminishing step sizes. 
Asymptotic convergence properties and  iteration-complexity bounds with and without convexity assumptions on the objective functions are stablished.  
Numerical experiments on some robust multiobjective optimization problems illustrating the practical behavior of the method are presented, and comparisons with the proximal gradient method~ \cite{TanabeFukudaYamashita2019} are discussed.

The organization of this paper is as follows. In Section~\ref{sec:Preliminares}, some notations, definitions, and auxiliary results used throughout of the paper are presented. Section~\ref{section of hipotesys} presents the assumptions on the considered multiobjective  composite optimization problem need to our analysis.  Moreover, we introduce the  gap function associated to problem~\eqref{eq:Problem1}  and study its  main  properties. In Section~\ref{algorithm}, we introduce a generalization of  the conditional gradient method  for solving  problem~\eqref{eq:Problem1}. We will also study  asymptotic convergence  properties and iteration-complexity bounds for the generated sequence by the proposed method. Numerical experiments are presented in Section~\ref{numerical}.  Finally, some conclusions are given in Section~\ref{conclusions}.

\section{Preliminaries} \label{sec:Preliminares}

 In this section, we present   some notations, definitions, and results  used throughout the paper.   We denote ${\mathbb{N}}=\{0,1,2,\ldots\}$ and  ${\mathbb{N}^*}=\{1,2,3, \ldots\}$. 
 $\R$, $\R_+$, and $\R_{++}$ are the set of real numbers, the set of nonnegative real numbers, and the set of positive real numbers, respectively.
 We define $\overline{\R}:=\mathbb{R} \cup \left\{ +\infty\right\}$.
 Likewise,  $\overline{\mathbb{R}}^m := (\mathbb{R} \cup \left\{ +\infty\right\})^m$.
Let ${\cal J}:=\{1,\ldots,m\}$, ${\mathbb R}^{m}_{+}:=\{u\in {\mathbb R}^m \mid   u_{j}\geq 0, \forall j\in {\cal J}\}$, and  ${\mathbb R}^{m}_{++}=\{u\in {\mathbb R}^m \mid u_{j}> 0, \forall j\in  {\cal J} \}$. For $u, v \in {\mathbb R}^{m}$, $v\succeq u$ (or $u \preceq v$) means that $v-u \in {\mathbb R}^{m}_{+}$ and $v\succ u$ (or $u \prec v$) means that  $v-u \in {\mathbb R}^{m}_{++}$. 
 The symbol $\langle \cdot, \cdot \rangle$ is the usual inner product in $\R^n$ and $\| \cdot \|$ denotes the Euclidean norm.  
  Let ${\cal C}\subset {\mathbb R}^{n}$ be a  convex set.  If  ${\cal C}$ is  compact,   its {\it diameter}  is the  finite number   $ \diam({\cal C}):= \max\left\{ \|x-y\| \mid \forall x, y\in {\cal C}\right\}$. 
    If $\K = \{k_1, k_2, \ldots \} \subseteq \N$, with $k_j < k_{j+1}$ for all $j\in \N$, then  we denote $\K  \subsetinf \N$. 
  The notation  $\sigma(t):= o(t)$ for $t\in \mathbb{R}/\{0\}$ means that  $\lim_{t\to 0}\sigma(t)/t=0$.

 The effective domain of  $\psi: \mathbb{R}^n \to \overline{\mathbb{R}}$ is defined as $ \mathrm{dom}(\psi):=\left\{ x\in \mathbb{ R}^n \mid \psi(x)< +\infty \right\}.$ The  function $\psi :\mathbb{R}^n \to  \overline{\mathbb{R}}  $ is said to be {\it Lipschitz continuous} with  constants $L>0$ on ${\cal C}\subset \mathrm{dom}(\psi)$ whenever  $|\psi(x)-\psi(y)|\leq L\|x-y\|$, for all $x, y\in {\cal C}.$
Let  $\psi: \mathbb{R}^n \to \overline{\mathbb{R}}$ be a convex function. The {\it directional derivative} of  $\psi$ at $x\in  \mathrm{dom}(\psi)$ in the direction $d\in \mathbb{R}^n$ is given by  $\psi^{\prime}(x;d):=\lim_{\alpha \to 0^+}({\psi(x+\alpha d)-\psi(x)})/{\alpha}$.  When $\psi$ is differentiable at $x\in  \mathrm{int}(\mathrm{dom}(\psi)) $, we can show that $\psi^{\prime}(x;d)=\left\langle \nabla \psi(x),d\right\rangle $. 
The next lemma is a well-known result in convex analysis whose proof can be found in \cite[Section 4.1]{bertsekas2003convex}.

\begin{lemma}\label{mono.der.direc.}
Let $\psi:\mathbb{R}^n \to \overline{\mathbb{R}} $ be a convex function. Then, the function  $\lambda \mapsto (\psi(x+\lambda d)-\psi(x))/\lambda,$  is non-decreasing in $(0,+\infty) $. In particular,  for all $\lambda \in(0,1]$, we have $ (\psi(x+\lambda d)-\psi(x))/{\lambda}\leq \psi(x+d)-\psi(x)$.  Consequently,  $\psi'(x; d)\leq \psi(x+d)-\psi(x).$
\end{lemma}

\begin{definition}
A  function $\psi:\mathbb{R}^n \to  \overline{\mathbb{R}}$  is upper semicontinuous (u.s.c) at $x\in\mathbb{R}^n$ if, for any sequence $ (x^k)_{k\in \mathbb{N}}$ converging to $x$, $\lim\sup_{k \to \infty} \psi(x^k)\leq \psi(x)$.  Likewise,  $\psi$ is said to be lower semicontinuous (l.s.c) at $x\in\mathbb{R}^n$ whenever $-\psi$ is (u.s.c)  at $x\in\mathbb{R}^n$
 or, equivalently, if  $\lim\inf_{k\to \infty}\psi(x^k) \geq \psi(x)$. 
 We say that $\psi$ is upper semicontinuous (resp. lower semicontinuous) if it is upper (resp. lower) semicontinuous at every point of its domain.
  \end{definition}

Let $\Psi :\mathbb{R}^n \to \overline{\mathbb{R}}^m$ be a vector-valued function with $\Psi(x):=(\psi_1(x),\dots,\psi_m(x))$ and  consider the problem 
\begin{equation}\label{eq:Problem2} 
\min_{x\in {\mathbb{R}^n}}{\Psi(x)}.
\end{equation}
 The {\it effective domain} of $\Psi$ is denoted by $\mathrm{dom}(\Psi):=\left\{ x\in \mathbb{R}^n \mid \psi_j(x)< +\infty, ~ \forall j\in\J \right\}$.  
 A point $x^* \in \mathbb{R}^n$ is called a  {\it Pareto optimal point} of   \eqref{eq:Problem2} if there exists no other $x\in \mathbb{R}^n$ such that $F(x)\preceq F(x^*)$ and $F(x) \neq F(x^*)$. In turn, $x^* \in \mathbb{R}^n$ is called a  {\it weakly Pareto optimal point} of   \eqref{eq:Problem2}, if   there exists no other $x \in \mathbb{R}^n$ such that $F(x) \prec F(x^*)$.  A {\it necessary optimality condition} for  problem \eqref{eq:Problem2} at a point ${\bar x}\in \mathbb{R}^n $  is given by 	
\begin{equation} \label{eq:CondCPoint22}
\max_{j\in {\cal J}}    \psi^{\prime}_{j}({\bar x}, d) \geq 0 , \quad \forall~d\in \mathbb{R}^n.
\end{equation} 
A point   $\bar{x}\in \mathrm{dom}(\Psi)$  satisfying  \eqref{eq:CondCPoint22} is  called  a {\it Pareto critical point} or a {\it stationary point} of problem~\eqref{eq:Problem2}.	
 Given $x, y \in\mathrm{dom}(\Psi)$, it is said that $x$ dominates $y$ when $\Psi(y)-\Psi(x)\in\R^m_+\setminus\{0\}$.
 The function $\Psi$ is said to be {\it convex} on ${\cal C}$ if 
$\Psi\left(\lambda x+(1- \lambda)y\right)\preceq \lambda \Psi(x)+(1- \lambda)\psi(y)$, for all  $x,y\in{\cal C}$, and all $\lambda\in[0,1], $ 
or equivalently,  if each component  $\psi_j$ of  $\Psi$, $j\in {\cal J}$, is a convex function on ${\cal C}$.   We recall that for  a differentiable convex function $\Psi$ on ${\cal C}$  we have   $\psi_{j}(y)-\psi_{j}(x) \geq \left\langle \nabla \psi_{j}(x),~y - x\right\rangle$, for all $x,y\in{\cal C}$ with $x\in  \mathrm{int}(\mathrm{dom}(\Psi))$, for all  $j \in \cal{J}.$ 
Next lemma shows that, in the convex case, the concepts of stationarity and weak Pareto optimality are equivalent, see \cite{Drummond2004}.

\begin{lemma} \label{l:convex}
If $\Psi$ is convex and $\bar{x}$ is a Pareto critical point, then $\bar{x}$ is a weakly Pareto optimal point of problem~\eqref{eq:Problem1}.
\end{lemma}

The next two lemma will be important for the convergence rate results.
\begin{lemma}\cite[Lemma 6]{polyak1987}\label{taxa2}
Let $ \left( a_k \right)_{k \in \mathbb{N}}$ be a nonnegative sequence of real numbers, if $\Gamma a_k ^2 \leq a_k - a_{k+1}$ for some $\Gamma>0$ and for any $k=1, \ldots, \ell$, then $a_\ell \leq {a_0}/({1+\ell \Gamma a_0}) <{1}/({\Gamma \ell})$.
\end{lemma}
\begin{lemma}\cite[Lemma 13.13]{Beck2017}\label{2taxe convv}
Let $ \left( a_k \right)_{k \in \mathbb{N}}$ and  $ \left( b_k \right)_{k \in \mathbb{N}}$ be nonnegative sequences of real numbers satisfying $a_{k+1} \leq a_{k} - b_{k} \beta_k +(A/2)\beta_{k}^2$  for all $k\in \N$,
where $\beta_k = 2/ (k+2)$ and $A$ is a positive number. Suppose that $a_k \leq b_k$ for all $k$. Then
\begin{itemize}
\item[(i)] $a_k \leq (2A)/k$ for all $k\in {\mathbb{N}^*}$; 
\item[(ii)]$\min_{\ell \in\{ \lfloor \frac{k}{2}\rfloor+2, \ldots,k \} }b_\ell \leq 8A/(k-2)$ for all $k=3,4,\ldots,$ where $\lfloor k/2 \rfloor = \max \left\{ n \in \mathbb{N}  \mid n \leq k/2\right\}. $
\end{itemize}
\end{lemma}
\section{The multiobjetive  composite optimization problem}\label{section of hipotesys}
Throughout our presentation, we  assume  that  $F:=(f_1, \ldots, f_m)$,  where  $f_{j}:=h_{j}+g_{j}$ for all $j\in {\cal J}:=\{1,2,\ldots,m\}$,  satisfies the following three conditions:
\begin{itemize}
	\item [{\bf (A1)}]  The function $h_j$ is differentiable, for all~ $j\in {\cal J}$; 
	\item [{\bf (A2)}] The function  $g_j$  is  proper, convex, and lower semicontinuous, for all~ $j\in {\cal J}$;
	\item [{\bf (A3)}] $\mathrm{dom}(G):=\left\{ x\in \mathbb{R}^n \mid g_j(x)< +\infty, ~j=1,2,\ldots,m \right\}$ is convex and  compact.
\end{itemize}
Since we are assuming  that $\mathrm{dom}(G)$ is  compact, for  future reference we take $\Omega>0$ satisfying 
\begin{equation} \label{eq:diam}
\Omega \geq \diam\left(\mathrm{dom}(G)\right).
\end{equation}

{\it We also consider the following three  additional assumptions, which will be considered only when explicitly stated}.
\begin{itemize}
	\item [{\bf (A4)}] The function  $g_j$  is Lipschitz continuous with constant  $L{g_j}>0$  in $ \mathrm{dom}(g_j)$, for all~ $j\in {\cal J}$;
	\item [{\bf (A5)}] The gradient  $\nabla h_j$ is Lipschitz continuous with  constants $L_j>0$,  for all~ $j\in {\cal J}$, and  
$$
L:=\max\{L_j \mid  j\in {\cal J}\};
$$
	\item [{\bf (A6)}] The function $h_j$, for all $ j \in \cal{J}$, is convex.
\end{itemize}
Before presenting the  method to solve problem~\eqref{eq:Problem1}, we first need  to  study  a gap function associated with this problem, which will play an important role in this work. This study  will be made in next section.

\subsection{The gap function }\label{functiongap}
This section is devoted to study the {\it gap function} $\theta:\mathrm{dom}(G) \to \R$ associated to problem~\eqref{eq:Problem1},  defined by 
\begin{equation} \label{eq:gapFunc}
\theta (x) := \min_{u \in{\mathbb{R}^n}}\max_{j\in {\cal J}}\left(  g_{j}(u) -g_{j}(x)+ \left\langle \nabla  h_j(x),  u-x \right\rangle  \right).
\end{equation}
As we prove below,  the gap function  $\theta (\cdot)$ will serve as a stopping criterion for the algorithm presented in the next section.
 We observe that, if the components of function  $ G $ are the indicator function of a set ${\cal C}$, the gap function $\theta (\cdot)$ in \eqref{eq:gapFunc} becomes   the one presented in \cite{Assunccao2021}.

Clearly, for each $x\in \mathrm{dom}(G)$,   the gap function $ \theta(x)$ is the optimum value  of the optimization  problem 
\begin{equation}\label{sb.prob.solu1}
\min_{u \in \mathbb{R}^n} \max_{j\in \cal{J}}\left( g_{j}(u)-g_{j}(x) +  \left\langle \nabla  h_j(x),  u-x \right\rangle\right).
\end{equation}
It follows from assumptions  (A1)--(A3) that problem~\eqref{sb.prob.solu1} has a solution (possibly not unique) and it belongs to $\mathrm{dom}(G)$. Thus,  we use  the notation $p(x)\in \mathrm{dom}(G)$ when referring to a solution of  problem~\eqref{sb.prob.solu1}, i.e., 
\begin{equation}\label{cnd.subp}
p(x) \in \arg\min_{u\in {\mathbb{R}^n}}\max_{j\in \cal{J}}\left( g_{j}(u)-g_{j}(x) +  \left\langle \nabla  h_j(x),  u-x \right\rangle\right).
\end{equation}
Therefore, combining  \eqref{eq:gapFunc} and \eqref{cnd.subp}, we conclude that
\begin{equation}\label{eq:OptVal}
\theta(x)=\max_{j\in {\cal J}}\left(  g_{j}( p (x) ) -g_{j}(x) + \left\langle \nabla  h_j(x),   p (x) -x  \right\rangle  \right),  \quad \forall~x\in \mathrm{dom}(G).
\end{equation}
To simplify the notations, for each  $x \in \mathrm{dom}(G)$ and $p (x)$ as in \eqref{cnd.subp}, we set  
 \begin{equation*} \label{eq:sdir}
d(x):=p(x)-x. 
\end{equation*} 
 In the following lemma, we show that $\theta (\cdot) $ can in fact be seen as a gap function for problem~\eqref{eq:Problem1}.
 
\begin{lemma}\label{optmalit.theta}
Let  $\theta: \mathrm{dom}(G) \to \mathbb{R}$ be defined  as in \eqref{eq:gapFunc}. Then
	\begin{itemize}
		\item [(i)] $\theta(x) \leq 0$, for all $x \in \mathrm{dom}(G)$;
		\item [(ii)] $\theta(x)=0$ if, and only if, $x$ is a Pareto critical point of problem~\eqref{eq:Problem1};
		\item [(iii)] $\theta (x)$ is upper semicontinuous. 
	\end{itemize}
\end{lemma}
\begin{proof}
Consider $(i)$ and let $x \in \mathrm{dom}(G)$.  The definition of $\theta (\cdot)$ in \eqref{eq:gapFunc} implies 
	\begin{eqnarray}\label{inq. theta}
	\theta(x) \leq \max_{j\in {\cal J}}\left(g_{j}(u)-g_{j}(x)+ \left\langle \nabla  h_j(x),  u-x \right\rangle \right), \quad \forall u \in \mathbb{R}^n.
	\end{eqnarray}
Thus, letting  $u=x$ in the previous inequality,   we conclude that  $\theta(x)\leq 0$, which proves $(i)$.  To prove item $(ii)$,  we first assume that $x$ is a Pareto critical point of problem~\eqref{eq:Problem1}. Therefore, by \eqref{eq:CondCPoint22}, we obtain
\begin{equation}\label{def.cri.paret.}
\max_{j\in \cal{J}}f_{j}^{\prime}(x;d)\geq0, \quad \forall d \in \mathbb{R}^n.
\end{equation}
Let $ d \in \mathbb{R}^n $  be arbitrary. Using   (A1) and (A2),  we have $f_{j}^{\prime}(x;d)=  g_{j}^{\prime}(x;d) + \left\langle \nabla h_{j}(x),d \right\rangle$. Thus,  it follows from \eqref{def.cri.paret.} that  $\max_{j\in {\cal J}}\{  g_{j}^{\prime}(x;d) +\left\langle \nabla  h_j(x), d \right\rangle\} \geq 0$, for all $d \in \mathbb{R}^n$. Hence,  by Lemma~\ref{mono.der.direc.}, we conclude that $\max_{j\in {\cal J}}\left\{  g_{j}(x+d) -g_{j}(x)+ \left\langle \nabla  h_j(x),  d \right\rangle \right\} \geq 0$. In particular,  letting  $ d = p (x) -x $, we have 
$$
\max_{j\in {\cal J}}\left(  g_{j}( p (x) ) -g_{j}(x) + \left\langle \nabla  h_j(x),   p (x) -x  \right\rangle \right) \geq 0.
$$	
Thus, using  \eqref{eq:OptVal},  we conclude that  $\theta(x)\geq 0$ which, together with item {\it (i)}, gives $\theta(x)=0$. 
Reciprocally, now we assume that $\theta(x)=0$. Thus, as in \eqref{inq. theta}, we obtain 
$$\max_{j\in {\cal J}}\left\{  g_{j}(u)-g_{j}(x) + \left\langle \nabla  h_j(x),  u-x \right\rangle \right\} \geq 0,  \quad \forall u \in \mathbb{R}^n.$$
In particular, letting $u=x+\alpha d$, for $\alpha>0$ and $ d \in \mathbb{R}^n$, we conclude that 
\begin{equation*}
\max_{j\in {\cal J}}\Big( \frac{g_{j}(x+\alpha d)-g_{j}(x)}{\alpha} + \left\langle \nabla  h_j(x), d \right\rangle \Big) \geq 0,  \quad \forall  \alpha>0, ~ \forall  d \in \mathbb{R}^n.
\end{equation*}
Since the  maximum function is continuous and $g_j$ has directional derivative at $x \in \mathrm{dom}(G)$,  we can take  limit as $\alpha$ goes to $0$  in the  last inequality to conclude that $\max_{j\in {\cal J}}\{ g_j'(x, d)+ \left\langle \nabla  h_j(x), d \right\rangle\}\geq 0$, for all $d \in \mathbb{R}^n$. Therefore,  \eqref{eq:CondCPoint22} holds  and thus  $x$ is a Pareto critical point of problem~\eqref{eq:Problem1}.  
We proceed to prove item $(iii)$. Let $x \in \mathrm{dom}(G)$ and consider a sequence $(x^k)_{k\in \mathbb{N}}$ such that $\lim_{k\to \infty}x^k =x$. Since $p(x) \in \mathrm{dom}(G)$, by  \eqref{eq:gapFunc}, we have 
 $$
 \theta(x^k)\leq \max_{j \in \cal{J}} \big(  g_{j}(p(x))-g_{j}(x^k) + \big\langle \nabla  h_j(x^k),  p(x) - x^k \big\rangle \big).
$$
Using the continuity of the maximum function and taking the upper limit in the  last inequality, we have
\begin{equation} \label{eq:scie}
\lim\sup_{k \to \infty} \theta(x^k) \leq  \max_{j \in \cal{J}} \big( g_{j}(p(x))+ \lim \sup_{k\to \infty}\big(-g_{j}(x^k)\big)+\big\langle \nabla  h_j(x),  p(x) - x \big\rangle\big). 
\end{equation}
On the other hand,  considering that  $g_j$  is lower semicontinuous in its effective domain, we obtain $\lim \sup_{k\to \infty}\left[-g_{j}(x^k)\right]\leq  - g_{j}(x)$. Therefore, combining this inequality with  \eqref{eq:scie} and \eqref{eq:OptVal},  we have $\lim \sup_{k \to \infty} \theta(x^k) \leq  \theta(x),$ which concludes the proof. 
\end{proof}

Hereafter, we denote:
$$e:=(1, \ldots, 1)^T\in {\mathbb R}^m.$$
When there is no confusion, we will also use letter $e$ to denote the column vector of ones with an alternative dimension.
In the following lemma, we present the counterpart of \cite[Lemma~1]{Assunccao2021} for $F$ being a composite function as defined in  \eqref{eq:Problem1}. Note that we assume that only the second component of $F$  has coordinates  with Lipschitz gradients. 

\begin{lemma}\label{lemma.desc.}
 Assume that  $F$ satisfies (A5).
Let  $x \in \mathrm{dom}(G)$ and $\lambda \in [0,1]$. Then
	\begin{equation}\label{eq. lemma 5}
	F(x+\lambda[p(x)-x]) \preceq F(x)+\Big( \lambda \theta (x)+\frac{L}{2}\|p(x)-x\|^2 \lambda^2\Big)  e.
	\end{equation}
\end{lemma}	
\begin{proof}
Let $j\in {\cal J}$. Since   $h_j$ has gradient Lipschitz continuous with constant $L_j$, $x \in \mathrm{dom}(G)$ and  $\lambda \in [0,1]$, we have
$$
f_j(x+\lambda[p(x) - x])\leq g_j((1-\lambda)x+\lambda g_j(p(x)) +  h_j(x)+\lambda\langle \nabla h_j(x), (p(x) -x)\rangle+\frac{L_j}{2}\|p(x) -x\|^2\lambda^2 . 
$$
Considering that $g_j$ is convex, we have $g_j((1-\lambda)x+\lambda p)\leq   (1-\lambda)g_j(x)+\lambda g_j(p)$. Thus,  combining this two previous inequalities, after some algebraic manipulations, we obtain 
$$
f_j(x+\lambda[p(x) - x])\leq f_j(x)+\lambda\left[ \langle \nabla h_j(x), (p(x) -x)\rangle -g_j(x)+g_j(p(x))\right] +\frac{L_j}{2}\|p(x) -x\|^2\lambda^2.
$$
Therefore, by \eqref{eq:OptVal} and due to $L=\max\{L_j:~j=1, \ldots m\}$, we have
$$
f_j(x+\lambda[p(x) - x])\leq f_j(x)+\lambda\theta(x) +\frac{L}{2}\|p(x) -x\|^2\lambda^2.
$$
Since the last inequality holds for all $j=1, \ldots, m$, then \eqref{eq. lemma 5} follows.
\end{proof}
\section{The generalized conditional gradient method}\label{algorithm}

In this section, we introduce a generalization of the conditional gradient method, also known as Frank-Wolfe algorithm, to solve multiobjective composite optimization problems.
We will also study asymptotic convergence properties and iteration-complexity bounds for the sequence generated by this method.
The analysis is carried out with three different step size strategies, namely, Armijo type, adaptive and diminishing step sizes. 
The conceptual method is described in Algorithm~\ref{Alg:CondG} below.

\bigskip
 \hrule\hrule \vspace{-5pt}
\begin{algorithm} 
{\bf Generalized  CondG method} \label{Alg:CondG} \\ 
\vspace{-8pt}\hrule\hrule
	\begin{description}	
\item[Step 0.] Choose $x^0\in $  $ \mathrm{dom}(G)$ and initialize  $k\gets 0$.
\item [Step 1.]  Compute an optimal solution $p(x^k)$ and the optimal value $\theta(x^k)$ as follows 
\begin{align}\label{eq:opt de p}
		p(x^k) &\in \arg\min_{u\in {\mathbb{R}^n}}\max_{j\in {\cal J}}\big{(}g_{j}(u)-g_{j}(x^k)+ \big\langle \nabla  h_j(x^k),  u - x^k \big\rangle \big{)}, \\
		\theta(x^k)&=\max_{j \in \cal{J}}\big{(} g_{j}(p(x^k))-g_{j}(x^k)+\big\langle \nabla  h_j(x^k),  p(x^k) - x^k \big\rangle  \big{)}. \label{eq:op.theta}
\end{align}
\item[ Step 2.]  If $\theta(x^k)= 0$, then {\bf stop}.    
\item[ Step 3.] Compute $\lambda_k \in (0, 1]$  and set 
\begin{equation}\label{eq:iteration}
x^{k+1}:=x^k+ \lambda_k (p(x^k)-x^k).
\end{equation}
\item[ Step 4.] Set $k\gets k+1$ and go to {\bf Step 1}.
	\end{description}
	\hrule
	\hrule
	\bigskip
\end{algorithm}

\begin{remark}
Let  $G: \mathbb{R}^n\to \overline{\mathbb{R}}^m$  be the indicator function of the set  ${\cal C}\subset \R^{n}$  in the multiobjective context, i.e., for all $j\in\J$, we have  $g_j(x)=0$ for all $ x \in {\cal{C}}$, and $g_j(x) =+\infty$  for all $ x \notin {\cal{C}}$. Then, assumptions (A2)--(A4) concerning $G$  are satisfied. Furthermore, Algorithm~\ref{Alg:CondG} merges into \cite[Algorithm~1]{Assunccao2021}.
 \end{remark}

As a consequence of Lemma~\ref{optmalit.theta},  Algorithm~\ref{Alg:CondG} successfully stops if a Pareto critical point is found. Thus, from now on, we assume,  without loss of generality, that  $\theta(x^k) <0$ for all $k\geq 0$ and therefore  an infinite sequence $(x^k)_{k \in \mathbb{N}}$ is generated by Algorithm~\ref{Alg:CondG}.
 We will analyze the generated sequence with three  step size strategies. We begin  by presenting  the Armijo-type step size.

\begin{Armijo}
	Let $\zeta \in(0,1)$ and $0<\omega_1<\omega_2<1$. The step size  $\lambda_k$ is  chosen  according the following line search algorithm:
\begin{description}
\item[Step LS0.] Set $\lambda_{k_{0}}=1$ and initialize $\ell \gets 0$.
\item[Step LS1.]   If 
$
F( x^k+ \lambda_{k_{\ell}} [p(x^k)-x^k])\preceq  F(x^{k}) - \zeta \lambda_{k_{\ell}} |\theta(x^k)| e,
$
then set    $\lambda_k:= \lambda_{k_{\ell}}$ and return to the main algorithm.
\item[Step LS2.]  Find $ \lambda_{k_{\ell+1}}\in [\omega_1  \lambda_{k_{\ell}} , \omega_2  \lambda_{k_{\ell}}]$, set $\ell \gets \ell + 1$,  and go to Step~LS1.
\end{description}
\end{Armijo}

The  second step size strategy is classical in the analysis of the scalar conditional gradient method, see for example \cite{BeckTeboulle2004}.

\begin{Adaptive}
Assume that  $F:=\left(f_1, \ldots, f_m\right)$ satisfies (A5) (and thus  \eqref{eq. lemma 5} in Lemma~\ref{lemma.desc.}). Define the step size as
\begin{equation}\label{eq:fixed.step}
\lambda_k:=\min\left\{1, \frac{|\theta(x^k)|}{L \|p(x^k)-x^k\|^2}\right\}=\argmin_{\lambda \in (0,1]}\Big( -|\theta(x^k)| \lambda+\frac{L }{2} \|p(x^k) -x^k\|^2 \lambda^2 \Big). 
\end{equation}
\end{Adaptive}
Since $\theta(x)<0$ and $p(x)\neq x$ for non-stationary points, the adaptive step size is well defined.  Next we present   the third  step size, which is well known   in the study of scalar conditional gradient method,   see for example \cite{Jaggi2013}.
\begin{diminishing}
Define the step size as
	\begin{equation} \label{eq:dimstep}
	\lambda_{k}:=\frac{2}{k+2}.
	\end{equation} 	
\end{diminishing}

\subsection{Convergence analysis  using Armijo step sizes} 
In this section, we analyze the sequence $(x^k)_{k \in \mathbb{N}}$ generated by Algorithm~\ref{Alg:CondG} with Armijo step sizes.  We begin by showing that the Armijo step  size strategy is well defined.  First, notice that assumptions (A2)--(A3) imply that  $p(x^k)  \in \mathrm{dom}(G)$  and   $\theta(x^k)$ in \eqref{eq:opt de p}   and  \eqref{eq:op.theta}, respectively,  are well defined. 
\begin{proposition}
	Let $\zeta \in (0,1)$, $x^k \in \mathrm{dom}(G)$, $p(x^k)$ and $\theta(x^k)$ as in \eqref{eq:opt de p} and \eqref{eq:op.theta},  respectively. Then, there exists $0<\bar{\eta}\leq1$ such that
	\begin{equation}\label{eq.de arm.}
		F(x^k+\eta[p(x^k)-x^k])\preceq F(x^k)-\zeta \eta|\theta(x^k)|e,  \quad \forall \eta \in (0,\bar{\eta}]. 
	\end{equation}
\end{proposition}
\begin{proof}
Since $H$ is differentiable, $G$ is convex,  $x^k  \in \mathrm{dom}(G)$, and $p(x^k) \in \mathrm{dom}(G)$, we conclude, for all $\eta \in (0,1)$, that
	\begin{eqnarray*}
	F(x^k+\eta[p(x^k)-x^k])&=&G(x^k+\eta[p(x^k)-x^k])+ H(x^k+\eta[p(x^k)-x^k])\nonumber\\
	&\preceq&  (1-\eta)G(x^k) +  \eta G(p(x^k))+H(x^k)+\eta JH(x^k)(p(x^k)-x^k)+ \frac{o(\eta)}{\eta}e .
	\end{eqnarray*}	
where $JH(x^k)$ denotes the Jacobian of $H$ at $x^k$. After some arrangement in the right hand side of the last inequality, we obtain 
\[F(x^k+\eta[p(x^k)-x^k])=  F(x^k)+\eta\Big( JH(x^k)(p(x^k)-x^k)+ G(p(x^k))-G(x^k)  \Big)+ \frac{o(\eta)}{\eta}e.\]
	Since $JH(x^k)(p(x^k)-x^k)\preceq \max_{j\in {\cal J}}\left\langle \nabla  h_j(x^k),  p(x^k)-x^k \right\rangle   e$, using \eqref{eq:op.theta}, we obtain  	
	\begin{equation*}\label{des.conv.diff}
	F(x^k+\eta[p(x^k)-x^k]) \preceq F(x^k)+\eta\zeta \theta(x^k)e +\eta\Big((1-\zeta)\theta(x^k)+\frac{o(\eta)}{\eta}\Big)e .  
	\end{equation*}
Terefore, considering that  $\theta(x^k) <0$, $\zeta \in (0,1)$, and  $\lim_{\eta \to 0}o(\eta)/\eta =0$, there exists $\bar{\eta}>0$ such that \eqref{eq.de arm.} holds for all $\eta \in (0,\bar{\eta}]$, concluding the proof. 
\end{proof}

In the following, we present our first asymptotic convergence result. It is worth noting that we only assume (A1)--(A3).

\begin{theorem}
Let $(x^k)_{k \in \mathbb{N}}$ be the sequence generated by Algorithm~\ref{Alg:CondG} with Armijo step sizes.
	Then, every limit point $\bar{x}$ of $(x^k)_{k \in \mathbb{N}}$ is a Pareto critical point for  problem~\eqref{eq:Problem1}.
\end{theorem}
\begin{proof}
Let $\bar{x} \in \mathrm{dom}(G)$ be a limit point of the sequence $(x^k)_{k \in \mathbb{N}}$ generated by Algorithm~\ref{Alg:CondG} and $ \mathbb{K} \subsetinf \mathbb{N} $ such that  $\lim_{k \in \mathbb{K}}x^k =\bar{x}$. 
It follows from the Armijo step size strategy that
	\begin{equation} \label{eq:faieq}
	0 \prec-\zeta \lambda_{k}\theta(x^k)e\preceq F(x^k)-	F(x^{k+1}),  \quad \forall k \in \mathbb{N},
	\end{equation} 
because	$\theta(x^k) <0$  for all $k\in \N$.
Consequently,  the sequence of functional values $(F (x^k)) _ {k \in  \mathbb{N}}$ is monotone decreasing. Moreover, since $F$ is continuous, we have $\lim_{k \in \mathbb{K}} F(x^{k}) = F(\bar{x})$.  Hence,   $(F (x^k)) _ {k \in  \mathbb{N}}$ converges  and   $ \lim_{k \in \mathbb{N}}[F(x^k)-F(x^{k+1})]=0.$ Thus,  \eqref{eq:faieq} implies    $\lim_{k \in \mathbb{N}} \lambda_{k} \theta(x^k)=0$ and, {\it a fortiori}, $\lim_{k \in \mathbb{K}} \lambda_{k} \theta(x^k)=0$. 
Therefore, there exists $\K_1\subsetinf\K$ such that at least one of the two following possibilities holds: $\lim_{k \in \mathbb{K}_1} \theta(x^k) =0$ or $\lim_{k\in \mathbb{K}_1} \lambda_{k} =0$. In case $\lim_{k \in \mathbb{K}_1} \theta(x^k) =0$, using Lemma~\ref{optmalit.theta}, we obtain $\theta(\bar{x}) = 0$, which implies that $\bar{x}$ is a Pareto critical point. Now consider the case $\lim_{k\in \mathbb{K}_1} \lambda_{k} =0$. Suppose by contradiction that  $ \theta (\bar {x}) < 0 $.   Since $\theta(\cdot)$ is  upper semicontinuous, $\lim_{k \in \mathbb{K}_1}x^k =\bar{x}$,  $\theta(\bar{x})<0$, and $\lim_{k\in \mathbb{K}_1} \lambda_{k} =0$,  there exist $\delta >0$ and  $\mathbb{K}_2 \subsetinf \mathbb{K}_1$ such that  
\begin{equation} \label{eq:ithetak}
\theta(x^k)<-\delta, \quad \forall  k\in \mathbb{K}_2, 
\end{equation} 
and also $\lambda_{k} <1 $ for all $ k\in \mathbb{K}_2$. Moreover, since $(p(x^k))_{k \in \mathbb{N}}\subset \mathrm{dom}(G)$ and  $\mathrm{dom}(G)$ is  compact, we assume,  without loss of generality, that there exists  $\bar{p} \in \mathrm{dom}(G)$ such that 
\begin{equation} \label{eq:limp}
\lim_{k\in \mathbb{K}_2} p(x^k) = \bar{p}.
\end{equation}  
Since $\lambda_{k} <1 $ for all $ k\in \mathbb{K}_2$, by the Armijo step size strategy, there exists $\bar{\lambda}_{k} \in (0,\lambda_k /\omega_1]$ such that
	$$
	F(x^k +\bar{\lambda}_{k}[p(x^k)-x^k]) \npreceq  F(x^k) + \zeta \bar{\lambda}_{k} \theta(x^k) e, \quad   \forall k\in \mathbb{K}_2,
	$$  
	which means that
	$$
	f_{j_k} ( x^k + \bar{\lambda}_{k}[p(x^k)-x^k]) > f_{j_k}(x^k) + \zeta \bar{\lambda}_{k} \theta(x^k), \quad   \forall k\in \mathbb{K}_2,
	$$
	for at least one $j_k \in \cal{J}$. Since $\cal{J}$ is finite set of indexes and  $\mathbb{K}_2$  is infinite, there exist $j^* \in \cal{J}$ and $\mathbb{K}_3\subsetinf \mathbb{K}_2$ such that
	\begin{equation}\label{eq.23}
	\displaystyle\frac{f_{j^*} ( x^k + \bar{\lambda}_{k}[p(x^k)-x^k])-f_{j^*}(x^k)  }{\bar{\lambda}_{k}}>\zeta  \theta(x^k),  \quad   \forall k\in \mathbb{K}_3.
	\end{equation}
On the other  hand,  owing to $0<\bar{\lambda}_{k}\leq 1$ and $g_{j^*}$ be  convex, we can apply  Lemma~\ref{mono.der.direc.} to obtain  
\begin{equation}\label{eq:f23}
\frac{g_{j^*}(x^k +\bar{\lambda}_{k}[p(x^k)-x^k])-g_{j^*}(x^k) }{\bar{\lambda}_{k}}\leq g_{j^*}(p(x^k))-g_{j^*}(x^k),  \quad   \forall k\in \mathbb{K}_3. 
\end{equation}
Moreover, due to $h$ be differentiable and $\lim_{k\in\K_3}\bar{\lambda}_{k}=0$,   we have, for all $k\in\K_3$,
\begin{equation}\label{eq:s23}
\bar{\lambda}_{k}  \langle \nabla h_{j^*}(x^k), p(x^k) - x^k \rangle =  h_{j^*} \left( x^k + \bar{\lambda}_{k}[p(x^k)-x^k]\right)-h_{j^*}(x^k)  - o(\bar{\lambda}_{k}\|p(x^k)-x^k\|). 
\end{equation}
Combining  \eqref{eq:op.theta} with \eqref{eq:f23} and \eqref{eq:s23}, after some algebraic manipulations, we obtain
	\begin{eqnarray} \label{equation24}
	\theta(x^k)&\geq& g_{j^*}(p(x^k))-g_{j^*}(x^k) + \langle \nabla h_{j^*}(x^k), p(x^k) - x^k \rangle \nonumber\\
	&\geq& \frac{f_{j^*} \left( x^k + \bar{\lambda}_{k}[p(x^k)-x^k]\right)-f_{j^*}(x^k)  }{\bar{\lambda}_{k} } - \frac{o(\bar{\lambda}_{k}\|p(x^k)-x^k\|)}{\bar{\lambda}_{k}},  \quad   \forall k\in \mathbb{K}_3. 
	\end{eqnarray}
Hence, \eqref{eq.23} and \eqref{equation24}  imply  that 
	\begin{equation}\label{eq.25}
	\frac{f_{j^*} \left( x^k + \bar{\lambda}_{k}[p(x^k)-x^k]\right)-f_{j^*}(x^k)  }{\bar{\lambda}_{k}}> \left(\frac{-\zeta}{1-\zeta} \right) \frac{o(\bar{\lambda}_{k}\|p(x^k)-x^k\|)}{\bar{\lambda}_{k}}, \quad   \forall k\in \mathbb{K}_3. 
	\end{equation}
On the other hand, it follows from \eqref{eq:ithetak} and \eqref{equation24} that 
	$$
	-\delta + \frac{o(\bar{\lambda}_{k}\|p(x^k)-x^k\|)}{\bar{\lambda}_{k}}> \frac{f_{j^*} \left( x^k + \bar{\lambda}_{k}[p(x^k)-x^k]\right)-f_{j^*}(x^k)  }{\bar{\lambda}_{k} },  \quad   \forall k\in \mathbb{K}_3. 
	$$
Combining  the last inequality with   \eqref{eq.25},  we have
	$$
	-\delta - \frac{o(\bar{\lambda}_{k}\|p(x^k)-x^k\|)}{\bar{\lambda}_{k}}>\left(\frac{-\zeta}{1-\zeta} \right) \frac{o(\bar{\lambda}_{k}\|p(x^k)-x^k\|)}{\bar{\lambda}_{k}}, \quad   \forall k\in \mathbb{K}_3.
	$$
Considering \eqref{eq:limp} and taking limits for $k\in\K_3$ on both sides of this inequality, we obtain  $-\delta \geq 0$, which is a contradiction  with  $\delta>0.$ Therefore,    $\theta(\bar{x}) =0$ and, from Lemma~\ref{optmalit.theta}~$(ii)$, we conclude that $\bar{x}$ is a Pareto critical point.
 \end{proof}
 
 To state  the following results, we introduce some  notations.  Since  $\mathrm{dom}(G)$ is a compact set and $\nabla h$ is continuous, we set 
\begin{equation} \label{eq:omegarho}
	  \rho:=\sup\{\|\nabla h_{j}(x)\| \mid x \in \mathrm{dom}(G),  ~j=1, \ldots, m\}.  
\end{equation} 
Moreover, considering that $g_j$ satisfies (A4),  we define 
\begin{equation} \label{eq:lm}
L_G:=\max \left\{L{g_j} \mid  j=1, \ldots, m\right\}>0, \quad \gamma:= \min \left\{  \frac{1}{(\rho + L_G) \Omega}, ~\frac{2\omega_1 (1-\zeta)}{L\Omega^2}\right\}.
\end{equation}

\begin{lemma}\label{le:bla}
Assume that $F$ satisfies (A4)--(A5). Then  $\lambda_{k} \geq   \gamma  |\theta(x^k)|>0$, for all $\in \N$.
\end{lemma}
\begin{proof}
Since $\lambda_k \in (0,1]$ for all $k\in \N$, let us consider two possibilities: $\lambda_k =1$ and $0<\lambda_k<1$.  First we assume that $\lambda_k =1$. 
It follows from \eqref{eq:op.theta} and Lemma~\ref{optmalit.theta}  that 
  $$
 \theta(x^k)=\max_{j\in {\cal J}}\left\{ g_{j}(p(x^k))-g_{j}(x^k)+ \langle \nabla  h_j(x^k),  p(x^k)-x^k \rangle \right\}<0,
  $$
  which implies that $  0< -\theta(x^k) \leq g_{j}(x^k)-g_{j}(p(x^k))+ \langle \nabla  h_j(x^k),  x^k-p(x^k) \rangle$ for all $j\in {\cal J}$.
  Thus, the Cauchy inequality  together  with (A4) and \eqref{eq:lm}  imply that 
$$
  0< -\theta(x^k) \leq  \left( L_G+ \|\nabla h_j(x^k) \|\right)\|p(x^k)-x^k\|. 
$$
 Using \eqref{eq:omegarho}, we have   $0< -\theta(x^k)  \leq  ( \rho +L_G) \Omega. $ Hence,  the definition of $\gamma$ in \eqref{eq:lm}  implies  that  
$$
0<-\gamma \theta(x^k)\leq \frac{-\theta(x^k)}{(\rho + L_G) \Omega}\leq 1, 
$$
which shows that  the desired  equality  holds for  $\lambda_k =1.$ Now, we assume  $0<\lambda_k <1$. Thus, from the Armijo step size strategy, we conclude that there  exist $ 0<\bar{\lambda}_k \leq \min\left\{1, \lambda_k / \omega_1\right\}$ and $j_k \in \cal{J}$, such that
$$
f_{j_k} (x^k +\bar{\lambda}_k[p(x^k)-x^k] ) > f_{j_k} (x^k )+ \zeta \bar{\lambda}_k \theta(x^k).
$$
On the other hand, by using Lemma~\ref{lemma.desc.},  we have
$$
f_{j} (x^k+\bar{\lambda}_k[p(x^k) - x^k]) \leq f_{j} (x^k)+\bar{\lambda}_k \theta(x^k)+\frac{L}{2}\|p^k -x^k\|^2 \bar{\lambda}_k
^2, \quad \forall j \in \cal{J}.
$$
Thus, combining the two previous inequalities with   $ 0<\bar{\lambda}_k \leq \min\left\{1, \lambda_k / \omega_1\right\}$,  we conclude that
$$
-\theta(x^k) (1-\zeta) < \frac{L}{2} \| p(x^k) - x^k\|^2 \bar{\lambda}_k \leq \frac{L}{2} \| p(x^k) - x^k\|^2 \frac{{\lambda}_k}{\omega_1}.
$$
Therefore, using the  definition of $\Omega$ in \eqref{eq:diam}  together with the definition of $\gamma$ in \eqref{eq:lm}, we obtain   
$$
0< -\gamma \theta(x^k)=-\frac{2\omega_1 (1-\zeta)}{L\Omega^2}\theta(x^k)  < \lambda_k, 
$$
which implies that desired inequality   also holds for  $0<\lambda_k <1$.
\end{proof}
In the following theorem,  we  obtain our first iteration-complexity bound. For that, we define
\begin{equation} \label{eq:finf}
f_{j^*}(x^0):= \max \{ f_{j}(x^0) \mid j \in {\cal{J}}\} \quad \mbox{and} \quad f_{j^*}^{\inf} := \min \{ f_{j}^* \mid j \in {\cal{J}}\},
\end{equation} 
where $f_j^* := \inf \{f_j(x) \mid x\in\mathrm{dom}(G) \}$ for all $j\in\J$.
\begin{theorem}\label{teorema para taxa2}
 Assume that $F$ satisfies (A4)--(A5).  Then $\displaystyle \lim_{k \to \infty } F(x^k)= F(x^*)$, for some $x^*\in \mathrm{dom}(G)$. Moreover, there hold:
\begin{itemize}
\item [i)] $\lim_{k\to \infty} \theta(x^k) = 0$;
\item [ii)]$\min \{ |\theta(x^k)| \mid  k=0,1,\ldots, N-1\} \leq \sqrt{[f_{j^*}(x^0) -f_{j^*}^{\inf} ]/ (\zeta \gamma N)}$.
\end{itemize}
\end{theorem}
\begin{proof}
By the Armijo step size strategy and  considering that   $\theta(x^k) < 0$ for all $k\in \N$, we have 
$F(x^k +\lambda_k[p(x^k) - x^k]) \preceq F(x^k) +\zeta\lambda_k \theta(x^k)e$
or, equivalently, $\zeta\lambda_k |\theta(x^k)|e  \preceq  F(x^k) - F(x^{k+1})$. Hence,   due to $\theta(x^k) < 0$, using  Lemma~\ref{le:bla}, we obtain 
\begin{equation}\label{des. theorm11}
0 \prec  \zeta\gamma |\theta(x^k)|^2e  \preceq  F(x^k) - F(x^{k+1}), 
\end{equation}
which implies that the sequence $( F(x^k))_{k \in \mathbb{N}}$ is monotone decreasing. On the other hand, since $(x^k)_{k \in \mathbb{N}}  \subset \mathrm{dom}(G)$ and $\mathrm{dom}(G)$ is compact, there exists  $x^* \in \mathrm{dom}(G)$ a limit point  of $(x^k)_{k \in \mathbb{N}}$. Let $(x^{k_j})_{j \in \mathbb{N}}$  be a subsequence of $(x^k)_{k \in \mathbb{N}} $ such that $\lim_{j \to \infty} x^{k_j}= x^*$.  Since   $(x^{k_j})_{j \in \mathbb{N}} \subset \mathrm{dom}(G)$  and $G$ satisfies (A4), we have 
$$
\|F(x^{k_j})  - F(x^*)\|=\|  G(x^{k_j})  - G(x^*) + H(x^{k_j})  - H(x^*)\|\leq  L_G \|x^{k_j} - x^*\|+  \|H(x^{k_j}) - H(x^*)\|,
$$
for all ${j \in \mathbb{N}}$. Considering that  $H$ is continuous and  $\lim_{j \to \infty} x^{k_j}= x^*$, we conclude from the last inequality that 
$ \lim_{j \to \infty } F(x^{k_j})= F(x^*).$ Thus, due to the monotonicity of the  sequence $( F(x^k))_{k \in \mathbb{N}}$,  we obtain  that   $\lim_{k \to \infty } F(x^{k})= F(x^*)$. Hence, taking limits on \eqref{des. theorm11}, we obtain  $\lim_{k\to \infty} |\theta(x^k)|^2 = 0$, which implies item~{\it (i)}.
By summing both sides of the second inequality in \eqref{des. theorm11} for $k=0,1,\ldots,N-1$ and using   \eqref{eq:finf}, we obtain
$$
\sum_{k=0}^{N-1} |\theta(x^k)|^2 \leq \frac{1}{\zeta \gamma} [ f_{j^*}(x^0) -  f_{j^*}^{\inf}].
$$
Thus, $\min \{ |\theta(x^k)|^2: k=0,1,\ldots, N-1\} \leq [f_{j^*}(x^0) -f_{j^*}^{\inf} ]/ (\zeta \gamma N)$, which implies the  item~{\it (ii)}.
 \end{proof}
 
 \begin{corollary} \label{cr:cosk}
 Assume that $F$ satisfies (A4)--(A5)  and $\varepsilon >0$. Define the set $K(\varepsilon) := \{ k\in \N : | \theta(x^k)| > \varepsilon \}.$ Then,
 $$
 |K(\varepsilon)| \leq \frac{f_{j_*} (x^0)-f_{j_*}^{\inf} }{\zeta \gamma} \frac{1}{\varepsilon^2},
 $$
 where $|K(\varepsilon)|$ denotes the number of elements of $K(\varepsilon)$.
 \end{corollary}
 \begin{proof}
 The proof follows straightforwardly from item $(ii)$ of Theorem~\ref{teorema para taxa2}.
 \end{proof} 
  
\begin{corollary}   \label{cr:coska}
Assume that $F$ satisfies (A4)--(A5)  and $\varepsilon>0$. Consider an iteration $k$ and let $F(x^k)$ be given. If  $|\theta(x^k)|>\varepsilon$, then the Armijo line search algorithm performs, at most, $1+ \ln(\gamma \varepsilon )/\ln(\omega_2)$ evaluations of $F$ to compute the step size $\lambda_k$.
\end{corollary}
\begin{proof}
Let  $\ell_k$ and $e(k)$ be, respectively, the number of inner iterations and  the  number of evaluations of $F$ in  the Armijo line search  algorithm  to compute $\lambda_k$. Then, by the definition of the algorithm, we have $e(k)= \ell_k + 1$ and   $\omega_2^{\ell_k}\geq \lambda_k$.  Hence, using  Lemma~\ref{le:bla}, it follows that $\omega_2^{\ell_k}\geq \gamma |\theta(x^k)|$. Since $|\theta(x^k)|>\varepsilon$, we have  $\omega_2^{\ell_k}\geq \gamma \varepsilon $. Therefore,  due to $0<w_2<1$, we obtain $\ell_k\leq \ln(\gamma \varepsilon )/\ln(\omega_2)$, concluding the proof.
\end{proof} 
\begin{theorem} \label{th:cbas}
Assume that $F$ satisfies (A4)--(A5)  and $\varepsilon>0$. Then, Algorithm~\ref{Alg:CondG} generates 
a point $x^k$ such that $|\theta(x^k)|\leq \varepsilon$, performing, at most,
$$
m\left[\left(1+ \frac{\ln(\gamma \varepsilon )}{\ln(\omega_2)}\right) \frac{ f_{j_*}(x^0)-f_{ j_*}^{\mathrm{inf}}}{\zeta \gamma } \frac{1}{\varepsilon^2}+1\right]={\cal O}(|\ln(\varepsilon)|\varepsilon^{-2})$$
evaluations  of functions  $f_1, \ldots f_m$, and 
$$
m\left[\frac{ f_{j_*}(x^0)-f_{ j_*}^{\mathrm{inf}}}{\zeta \gamma } \frac{1}{\varepsilon^2}+1\right]={\cal O}(\varepsilon^{-2})
$$
evaluations  of gradients  $\nabla h_1, \ldots, \nabla h_m$.
\end{theorem}
\begin{proof}
The proof follows from the combination of  Corollaries~\ref{cr:cosk}  and \ref{cr:coska}. 
\end{proof} 
Similar results of Corollaries \ref{cr:cosk} and \ref{cr:coska} and Theorem~\ref{th:cbas}, with respect to the scalar  gradient method   were obtained in \cite{GrapigliaSachs2017}.  

 \begin{remark}
 Assume that   $F$ satisfies (A4)--(A6).  Moreover  assume that $\lim_{k \to \infty}F(x^k)=F(x^*)$,  for some $x^*\in \mathrm{dom}(G)$,  and take $\Omega>0$ satisfying \eqref{eq:diam}. Then, for all $k\in {\mathbb{N}^*}$,  the following inequality holds 
   \begin{equation}\label{des.teo.mos.1}
   \min_{j \in \cal{J}}\left( f_{j}(x^k) - f_{j}(x^*)\right) \leq \frac{1}{\zeta \gamma} \frac{1}{ k},
      \end{equation}
      where $\gamma$ is given in \eqref{eq:lm}.
 Indeed, since $\lambda_k$ satisfies the Armijo step size rule, we have  
$
F(x^{k+1}) - F(x^*)\preceq F(x^k) - F(x^*)+\zeta\lambda_k \theta(x^k)e. 
$
 Thus,  the last inequality, together with Lemma~\ref{le:bla}, implies
\begin{equation}\label{eq. teorm.}
 \min_{j \in \cal{J}}\left( f_j(x^{k+1}) -f_j(x^*) \right) \leq  \min_{j \in \cal{J}}\left( f_j(x^{k}) -f_j(x^*) \right) -\zeta \gamma\theta(x^k)^2.
\end{equation}
On the other hand, using  the convexity of $h_j$, for all $ j \in \cal{J}$,  we conclude   that 
$$
f_j(x^*)  - f_j(x^k)=  g_j(x^*)-g_j (x^k) + h_j(x^*)-h_j (x^k)  \geq  g_j(x^*)-g_j (x^k)+\big\langle  \nabla h_j (x^k ) , x^* -x^k\big\rangle, 
$$
Since  $\left( F(x^k)\right)_{k \in \mathbb{N}}$ is  decreasing monotone  and  $\lim_{k \to \infty}F(x^k)= F(x^*)$,
we have $ F(x^*) \preceq F(x^k)$, for all $k\in \N$. Thus,  the last inequality implies that  
$$
0\geq f_j(x^*)  - f_j(x^k) \geq  g_j(x^*)-g_j (x^k)+ \langle  \nabla h_j (x^k ) , x^* -x^k\rangle, \quad \forall j \in \cal{J}.
$$
Taking maximum in the least inequality and using  the definition of $\theta(x^k)$ in \eqref{eq:op.theta}, we conclude that 
$$
0 \geq \max_{j \in \cal{J}}\left(f_{j}(x^*)-f_{j}(x^k) \right) \geq \max_{j \in \cal{J}} \left(  g_j(x^*)-g_j (x^k)+\langle  \nabla h_j (x^k ) , x^* -x^k\rangle \right) \geq  \theta(x^k), 
$$
which implies that $0 \geq -\min_{j \in \cal{J}}\{ f_{j}(x^k)-f_{j}(x^*)\}  \geq  \theta(x^k)$.  Therefore,   we obtain  
$$
0 \leq \Big(\min_{j \in \cal{J}}\big( f_{j}(x^k)-f_{j}(x^*)\big)\Big)^2  \leq  \theta(x^k)^2.
$$
The combination of the last inequality with \eqref{eq. teorm.} yields  
$$
 \zeta \gamma \Big(\min_{j \in \cal{J}}\big( f_{j}(x^k)-f_{j}(x^*)\big)\Big)^2\leq  \min_{j \in \cal {J}}\big(f_{j}(x^k)-f_{j}(x^*)\big) -  \min_{j \in \cal {J}}\big(f_{j}(x^{k+1})-f_{j}(x^*)\big), 
$$
for all $k \in \mathbb{N}$. Finally, applying  Lemma~\ref{taxa2}, with $a_{k}=\min_{j \in\cal{J}}\left(f_{j}(x^k)-f_{j}(x^*)\right)$ and $\Gamma=\zeta \gamma$, we obtain the desired inequality  \eqref{des.teo.mos.1}.  
 \end{remark}

\subsection{Convergence analysis using adaptive and diminishing  step sizes} 
The purpose of  this section is to  analyze the sequence $(x^k)_{k \in \mathbb{N}}$ generated by  Algorithm~\ref{Alg:CondG} with adaptive and  diminishing   step  sizes.   We begin by showing that, in particular,  if $(x^k)_{k \in \mathbb{N}}$ is generated by    Algorithm~\ref{Alg:CondG} with the adaptive step size \eqref{eq:fixed.step}, then $(F(x^k))_{k \in \mathbb{N}}$ is a nonincreasing  sequence. 
\begin{lemma}\label{lemma descente}
Assume that   $F$ satisfies (A5). Let  $(x^k)_{k \in \mathbb{N}}$ be  generated by Algorithm~\ref{Alg:CondG} with the adaptive step size \eqref{eq:fixed.step}. Then
	\begin{equation}\label{des.mon.3}
	F(x^{k+1})- F(x^k) \preceq - \frac{1}{2}\min_{}\left\{ |\theta(x^k)|, \frac{\theta(x^k)^2}{L \Omega^2}\right\} e, \quad \forall k \in \mathbb{N},
	\end{equation}	
	where $\Omega>0$ is given in \eqref{eq:diam}. As a consequence, $(F(x^k))_{k \in \mathbb{N}}$ is a nonincreasing   sequence.	
	\end{lemma}
\begin{proof}
Let us analyze the two possibilities   for $\lambda_{k}$ defined in \eqref{eq:fixed.step}. First, assume that $\lambda_{k}=1$. In this case,  using \eqref{eq:fixed.step}, we have  $L\|p(x^k)-x^k\|^2 \leq |\theta(x^k)|$. Thus,  taking into account (A5), we apply  Lemma~\ref{lemma.desc.} with $\lambda=1$ and  $x=x^k$ to obtain 
 \begin{equation} \label{des. mon. ilemma2}
	F(p(x^k)) \preceq F(x^k) + \Big( \theta(x^k)+\frac{L}{2}\|p(x^k)-x^k\|^2 \Big)e\preceq F(x^k) + \Big( \theta(x^k)+\frac{1}{2} |\theta(x^k)| \Big)e.
 \end{equation}
 Due to $\lambda_k=1$, it follows from     \eqref{eq:iteration}  that $x^{k+1}=p(x^k)$. Therefore, since  $|\theta(x^k)|=-\theta(x^k)$,  we conclude  from \eqref{des. mon. ilemma2} that 
	\begin{equation}\label{des. mon. lemma2}
	F(x^{k+1}) \preceq F(x^k) - \frac{1}{2}|\theta(x^k)|e.
	\end{equation}
	Now, assume that $\lambda_{k}=-\theta(x^k)/(L\|p(x^k)-x^k\|^2)$.   Thus,  applying  Lemma~\ref{lemma.desc.} with $\lambda=\lambda_k$ and  $x=x^k$,  and considering  \eqref{eq:iteration}, we  obtain 
	\begin{equation}\label{des. mon. lemma21}
	F(x^{k+1}) \preceq F(x^k) + \left(\lambda_{k} \theta(x^k)+\frac{L}{2}\|p(x^k)-x^k\|^2 \lambda_{k}^2\right)e = F(x^k) - \frac{\theta(x^k)^2}{2L\|p(x^k)-x^k\|^2}e.
	\end{equation}
	Therefore,  the combination of \eqref{des. mon. lemma2} and \eqref{des. mon. lemma21} yields
	$$
		F(x^{k+1}) \preceq F(x^k) -\frac{1}{2}\min \left\{  |\theta(x^k)|,\frac{\theta(x^k)^2}{L\|p(x^k)-x^k\|^2} \right\} e. 
        $$
	Since $\Omega \geq \diam(\mathrm{dom}(G))$, the last inequality implies that  \eqref{des.mon.3} holds.
\end{proof}

For state the next result, we consider constants $f_{j^*}(x^0)$ and  $f_{j^*}^{\inf}$ as in \eqref{eq:finf}.

\begin{proposition}\label{proposition4}
Assume that   $F$ satisfies (A5).  Let $(x^k)_{k \in \mathbb{N}}$ be generated by Algorithm~\ref{Alg:CondG} with the adaptive step size \eqref{eq:fixed.step}. Then
	\begin{itemize}
		\item [(i)] $\lim_{k \to\infty}\theta(x^k)=0$;
		
		\item[(ii)] for every $N \in \mathbb{N}$, there holds
		$$
		\min \left\{ |\theta(x^k)| \mid k=0,1,\ldots,N-1\right\} \leq \max\left\{ \frac{2}{N}\left(f_{j^*}(x^0)-f_{j^*}^{\inf}\right), \Omega \sqrt{\frac{2L}{N}\left(f_{j^*}(x^0)-f_{j^*}^{\inf}\right)}\right\}.
		$$
	\end{itemize} 
\end{proposition}
\begin{proof}
	Lemma~\ref{lemma descente} implies that
	\begin{equation}\label{eq:conv.theta}
	0< \min \left\{ |\theta(x^k)|,\frac{\theta(x^k)^2}{L\Omega^2} \right\}e  \leq 2\left(F(x^k)-F(x^{k+1})\right), \quad \forall k\in \N. 
	\end{equation}
	As in the proof of Lemma~\ref{teorema para taxa2}, it follows that $(F(x^k))_{k \in \mathbb{N}}$ converges. Thus, taking limits as $k$ goes to infinity on \eqref{eq:conv.theta}, we obtain  {\it (i)}.
Next we proceed to prove {\it (ii)}. By summing both sides  of the second inequality in \eqref{eq:conv.theta} for $k=0,1,\ldots,N-1$,  and taking into account the definition of $f_{j^*}(x^0)$ and  $f_{j^*}^{\inf}$, we obtain
	$$
	\sum_{k=0}^{N-1} \min\left\{ |\theta(x^k)|, \frac{\theta(x^k)^2}{L\Omega^2} \right\} \leq  2\left(f_{j^*}(x^0)-f_{j^*}^{\inf}\right). 
	$$
Therefore, we have 
$$
		\min \left\{  \min\left\{ |\theta(x^k)|, \frac{\theta(x^k)^2}{L\Omega^2} \right\}:~ k=0,1,\ldots,N-1\right\} \leq  \frac{2}{N}\left(f_{j^*}(x^0)-f_{j^*}^{\inf}\right), 
		$$
which implies  the statement of item {\it (ii)}.  
\end{proof}
\begin{theorem}
Assume that   $F$ satisfies (A5).  Let $ (x^k)_{k \in \mathbb{N}}$ be generated by Algorithm~\ref{Alg:CondG} with the adaptive step size.  Then, every limit point of $ (x^k)_{k \in \mathbb{N}}$ is a Pareto critical point of  problem~\eqref{eq:Problem1}.
\end{theorem}
\begin{proof}
Let $\bar{x}$ be a limit point of the sequence $(x^k)_{k \in \mathbb{N}}$ and $ \mathbb{K} \subsetinf \mathbb{N}$ such that  $\lim_{k \in \mathbb{K}}x^k =\bar{x}.$ Since $x^k\in \mathrm{dom}(G)$ for all $k\geq 0$, and   due to $ \mathrm{dom}(G)$ be compact, we conclude   that   $\bar{x}\in \mathrm{dom}(G)$.  On the other hand, Proposition~\ref{proposition4}{\it (i)}  implies that  $\lim_{k \in \mathbb{K}}\theta(x^k)=0$. Hence, considering that $\lim_{k \in \mathbb{K}}x^k =\bar{x}$,  it follows from   Lemma~\ref{optmalit.theta}{\it (iii)} that  $0\leq \theta(\bar{x})$. Thus, owing  to $\bar{x}\in \mathrm{dom}(G)$,  Lemma~\ref{optmalit.theta}{\it (i)} implies $\theta(\bar{x})=0$. Therefore, applying Lemma~\ref{optmalit.theta}{\it (ii)}, we conclude   that $\bar{x}$ is a  Pareto critical point of  problem~\eqref{eq:Problem1}. 
\end{proof}
 \begin{theorem}\label{teorema taxa3}
Assume that   $F$ satisfies (A5)--(A6).  Let $ (x^k)_{k \in \mathbb{N}}$ generated by Algorithm~\ref{Alg:CondG} with $\lambda_k$ satisfying the adaptive or the diminishing step size, i.e., \eqref{eq:fixed.step} or \eqref{eq:dimstep}.  Assume that  $\lim_{k \to \infty}F(x^k)=F(x^*)$,  for some $x^*\in \mathrm{dom}(G)$. Then
 	$$
 	\min_{\ell \in \left\{ \lfloor \frac{k}{2}\rfloor+2, \cdots,k  \right\} } |\theta(x^{\ell})| \leq \frac{8L\Omega^2}{k-2}, \quad k=3,4,\ldots.
 	$$
 \end{theorem}
\begin{proof}
We first claim that
 \begin{equation}\label{eq:2addim}
F(x^{k+1}) \preceq F(x^k)+\Big( \beta_k \theta (x^k)+\frac{L}{2}\|p(x^k)-x^k\|^2 \beta_k^2\Big)  e,
\end{equation}
where $\beta_k:=2/(k+2)$. Indeed, by applying  Lemma~\ref{lemma.desc.} with $x=x^k$ and $\lambda=\lambda_k$,  we have 
 \begin{equation}\label{eq:3addim}
F(x^k+\lambda_k[p(x^k)-x^k]) \preceq F(x^k)+\Big( \lambda_k \theta (x^k)+\frac{L}{2}\|p(x^k)-x^k\|^2 \lambda_k^2\Big)  e.
\end{equation}
If $\lambda_k$ is the diminishing step size given in \eqref{eq:dimstep}, then \eqref{eq:2addim} and \eqref{eq:3addim} trivially coincide.
We now assume that $\lambda_k$ is the adaptive step size given in \eqref{eq:fixed.step}. Since $\beta_k \in (0,1]$, it follows from  \eqref{eq:fixed.step} that
$ \lambda_k \theta(x^k)+(L/2)  \|p(x^k) -x^k\|^2 \lambda_k^2 \leq \beta_k\theta(x^k)+(L/2) \|p(x^k) -x^k\|^2 \beta_k^2$. The latter inequality together with \eqref{eq:3addim} yields \eqref{eq:2addim}.
Therefore, \eqref{eq:2addim} holds for both adaptive and diminishing strategies.
Now, by \eqref{eq:2addim} and  taking into account  that  $\|p(x^k) - x^k\| \leq \Omega$, we have
	\begin{equation} \label{eq:fici}
	\min_{j \in \cal{J}} \big( f_{j}(x^{k+1}) - f_{j}(x^*)\big) \leq 	\min_{j \in \cal{J}} \big( f_{j}(x^{k}) - f_{j}(x^*)\big)+\beta_k  \theta(x^k)  + \frac{L}{2}\Omega^2\beta_k  ^2, \quad \forall n\in\N.
		\end{equation}
	On the other hand,  due to  $\lim_{k \to \infty}F(x^k)=F(x^*)$ and using   (A6),  we  obtain
	$$
	0 \geq f_{j}(x^*) -f_{j}(x^k) \geq \big\langle \nabla h_j(x^k), x^* - x^k \big\rangle + g_{j}(x^*) -g_{j}(x^k), \quad k \in \N.
	$$
	Thus, taking the maximum and using the optimality of $p(x^k)$ in \eqref{eq:opt de p}, we have 
	$$
	0 \geq \max_{j \in \cal{J}}\big( f_{j}(x^*)-f_{j}(x^k) \big) \geq \max_{j \in \cal{J}} \big( \big\langle  \nabla h_j (x^k ) , x^* -x^k\big\rangle+ g_j(x^*)-g_j (x^k) \big) \geq  \theta(x^k), 
	$$
	which implies  $0 \leq \min_{j \in \cal{J}}\{ f_{j}(x^k)-f_{j}(x^*)\}  \leq  |\theta(x^k)|$. Therefore, using \eqref{eq:fici},  we can apply Lemma~\ref{2taxe convv}{\it (ii)} with $a_k = \min_{j \in \cal{J}}\{ f_{j}(x^k)-f_{j}(x^*)\}$, $b_k =| \theta(x^k)| $,
and $A=L \Omega^2$ to obtain the desired inequality.	
\end{proof}

\begin{remark} 
Assume that   $F$ satisfies (A5)--(A6).  Let $ (x^k)_{k \in \mathbb{N}}$ generated by Algorithm~\ref{Alg:CondG} with $\lambda_k$ satisfying the  adaptive or the diminishing step  sizes, i.e., \eqref{eq:fixed.step} or \eqref{eq:dimstep}.   Suppose that  $\lim_{k \to \infty}F(x^k)=F(x^*)$,  for some $x^*\in \mathrm{dom}(G)$. Then
\begin{equation}\label{taxa-- min.}
\min_{j \in \cal{J}}\left( f_{j}(x^k) - f_{j}(x^*)\right) \leq \frac{2L\Omega^2}{k},   \quad  \forall k \in \N.
\end{equation}
Indeed, by using the same argument of the proof of Theorem~\ref{teorema taxa3} and Lemma~\ref{2taxe convv}{\it (i)}, the inequality \eqref{taxa-- min.} follows.
  \end{remark}  
  
\section{Numerical experiments} \label{numerical}

This section presents some numerical experiments in order to illustrate the applicability of our approach. For this aim, we compare:
\begin{itemize}
\item the Generalized Conditional Gradient method (Algorithm~\ref{Alg:CondG});
\item the Proximal Gradient method proposed in \cite{TanabeFukudaYamashita2019}.
\end{itemize}
 We implemented both methods using the Armijo step size strategy with parameters $\zeta=10^{-4}$, $\omega_1=0.05$, and $\omega_2=0.95$. 
 Without attempting to go into details, we remark that the Armijo line search was coded based on quadratic polynomial interpolations of the coordinate functions, see \cite{PerezPrudente2019} for line search strategies in the vector optimization setting. 
The main difference between the two considered methods consists of the subproblem to be solved to calculate the search direction. 
While for the Generalized Conditional Gradient method the subproblem is given in \eqref{eq:opt de p}, in the Proximal Gradient method the search direction in iteration $k$ is defined by $d_{PG}(x^k):=p_{PG}(x^k)-x^k$, where
\begin{equation}\label{PGsub}
p_{PG}(x^k) \in \arg\min_{u\in {\mathbb{R}^n}}\max_{j\in {\cal J}}\big{(}g_{j}(u)-g_{j}(x^k)+ \big\langle \nabla  h_j(x^k),  u - x^k \big\rangle +\frac{\mu}{2}\|u - x^k\|^2\big{)},
\end{equation}
and $\mu>0$ is an algorithmic parameter. In our experiments, we set $\mu:=1$. In this case, when $G(x)\equiv 0$,
\eqref{PGsub} reduces to the classical steepest descent approach proposed in \cite{FliegeSvaiter2000}.
We denote the optimal value of problem \eqref{PGsub} by $\theta_{PG}(x^k)$.
As in Lemma~\ref{optmalit.theta}, $\theta_{PG}(\cdot)$ can be used to characterize Pareto critical points, see \cite{TanabeFukudaYamashita2019}.
In order to standardize the stopping criteria, all runs were stopped at an iterate $x^k$ declaring convergence if
\begin{equation} \label{sc}
\frac{\|x^k-x^{k-1}\|_{\infty}}{\max\{1,\|x^{k-1}\|_{\infty}\}}\leq 10^{-4} \quad \mbox{and}\quad \abs[1]{\theta_{\textrm{PG}}(x^k)} \leq 10^{-4}.
\end{equation}
The first criterion in \eqref{sc} seeks to detect the convergence of the sequence $\{x^k\}$, while the second guarantees to stop at an {\it approximately} Pareto critical point. For Algorithm~\ref{Alg:CondG}, we only calculate $\theta_{\textrm{PG}}(x^k)$ when the first criterion in  \eqref{sc} is satisfied. We also consider a stopping criterion related to failures: the maximum number of allowed iterations was set to 200. 
The codes are written in Matlab and are freely available at \url{https://github.com/lfprudente/CompositeMOPCondG}.

\vspace{12pt}
\noindent{\bf Set of test problems:}
The set of test problems is related to robust multiobjective optimization.
Robust optimization deals with uncertainty in the data of optimization problems, in such a way that the optimal solution must occur in the worst possible scenario, i.e., for the worst possible value that the uncertain data can assume.
Let us discuss how test problems were designed.
The differentiable part $H$ that makes up the objective function $F$ comes from some multiobjective problem found in the literature.
Table~\ref{tab:problems} shows the main characteristics of the chosen problems.
The first two columns identify the name of the problem and the corresponding reference where its formulation can be found.
Columns ``$n$'' and ``$m$'' inform the numbers of variables and objectives of the problem, respectively. 
Column ``Convex'' indicates whether the corresponding function $H$ is convex or not.
For each test problem, we denote the uncertainty parameter by $z\in\R^n$ and assume that its information is incorporated into function $G$.
 For each  $j \in {\cal J}$, we assume that $\mathrm{dom}(g_j)=\{x\in\R^n \mid lb \preceq x\preceq ub\}=:\cal{C}$, where $lb,ub\in\R^n$ are given in the last columns of Table~\ref{tab:problems}, and define $g_{j}:\cal{C}\to\R$ by
\begin{equation}\label{gdef}
g_{j}(x):=\displaystyle \max_{z \in {\cal{Z}}_j}   \langle x,z \rangle,
\end{equation}
 where ${\cal{Z}}_{j} \subset \R^n$ is the {\it uncertainty set}. 
 Let $B_j \in \R^{n\times n}$ a nonsingular matrix and $\delta > 0$ be given. 
We set  
\begin{equation}\label{Zdef}
{\cal{Z}}_{j} := \{ z \in \R^n \mid  -\delta e \preceq B_jz\preceq \delta e\},
\end{equation}
where  $e=(1,\ldots,1)^T\in\R^n$.
Since ${\cal{Z}}_{j}$ is a nonempty and compact, $g_{j}(x)$ is well-defined.
It is easy to see that  $g_j$ satisfies (A2)--(A4).
Note that parameter $\delta$ controls the uncertainty of the problem.
In our tests, the elements of the matrix $B_j$ were randomly chosen between 0 and 1.
In turn, given an arbitrary point $\bar x\in \cal{C}$, parameter $\delta$ was taken as
\begin{equation}\label{deltadef}
\delta := \bar{\delta} \|\bar x\|,
\end{equation}
 where $0.02\leq \bar{\delta} \leq 0.10$ was also chosen at random.
We mention that the definition of the non-differentiable function $G$ in \eqref{gdef}--\eqref{Zdef} has appeared in \cite{TanabeFukudaYamashita2019}.
Other works dealing with robust multiobjective optimization problems include \cite{EhrgottIdeSchobel2014, FliegeWerner2014, IdeSchobel2016}.

\begin{table}[h!]
\centering
\begin{threeparttable}{\fontsize{7}{7.1}\selectfont
\rowcolors{2}{}{lightgray}
\begin{tabular}{|ccccccc|} 
\hline
\rowcolor[gray]{.90} Problem & Ref. & $n$  & $m$ & Convex & $lb$ & $ub$ \\ 
\hline              
AP1&  \cite{ansary} &     2 &  3 & Y  & $(-10,-10)$ & $(10,10)$\\
AP2&  \cite{ansary} &     1 &  2 & Y  & $-100$ & $100$\\
AP3&  \cite{ansary} &     2 &  2 & N  & $(-100,-100)$ & $(100,100)$\\
AP4&  \cite{ansary}  &     3 &  3 & Y  & $(-10,-10,-10)$ & $(10,10,10)$\\
BK1&  \cite{reviewproblems}  &  2 &  2 & Y & $(-5,-5)$  & $(10,10)$   \\ 
DD1&  \cite{dd1} &  5 &  2 & N  & $(-20,\ldots,-20)$ & $(20,\ldots,20)$\\
DGO1&  \cite{reviewproblems}  &  1 &  2 & N  & $-10$  & $13$   \\ 
DGO2&  \cite{reviewproblems}  &  1 &  2 & Y  & $-9$  & $9$   \\ 
FA1&  \cite{reviewproblems}  &  3 &  3 & N  & $(0.01,0.01,0.01)$  & $(1,1,1)$   \\ 
Far1&  \cite{reviewproblems}  &  2 &  2 & N  & $(-1,-1)$  &  $(1,1)$\\ 
FDS &  \cite{Fliege.etall2009}  & 5 & 3 &Y& $(-2,\ldots,-2)$ &  $(2,\ldots,2)$  \\  
FF1 & \cite{reviewproblems} &     2 &  2 & N  & $(-1,-1)$ &  $(1,1)$ \\ 
Hil1&  \cite{hill} &     2 &  2 & N  & $(0,0)$  &   $(1,1)$\\ 
IKK1&  \cite{reviewproblems}  &  2 &  3 & Y  & $(-50,-50)$  & $(50,50)$   \\ 
IM1&  \cite{reviewproblems}  &  2 &  2 & N  & $(1,1)$  & $(4,2)$   \\ 
JOS1&  \cite{10.5555/2955239.2955427} &  100 &  2 & Y  & $(-100,\ldots,-100)$ & $(100,\ldots,100)$\\
JOS4&  \cite{10.5555/2955239.2955427}  &  100 &  2 & N  & $(-100,\ldots,-100)$ & $(100,\ldots,100)$\\
KW2 & \cite{kw2} &  2 &  2 & N  & $(-3,-3)$  &  $(3,3)$ \\ 
LE1&  \cite{reviewproblems}  &  2 &  2 & N  & $(1,1)$  & $(10,10)$   \\ 
Lov1& \cite{doi:10.1137/100784746} & 2 & 2 & Y & $(-10,-10)$ & $(10,10)$\\
Lov2& \cite{doi:10.1137/100784746} & 2 & 2 & N & $(-0.75,-0.75)$ & $(0.75,0.75)$\\
Lov3& \cite{doi:10.1137/100784746} & 2 & 2 & N & $(-20,-20)$ & $(20,20)$\\
Lov4& \cite{doi:10.1137/100784746} & 2 & 2 & N & $(-20,-20)$ & $(20,20)$\\
Lov5& \cite{doi:10.1137/100784746} & 3 & 2 & N & $(-2,-2,-2)$ & $(2,2,2)$\\
Lov6& \cite{doi:10.1137/100784746} & 6 & 2 & N & $(0.1,-0.16,\ldots,-0.16)$ & $(0.425,0.16,\ldots,0.16)$\\
LTDZ&  \cite{laumanns2002combining} &  3 &  3 & N & $(0,0,0)$  & $(1,1,1)$   \\ 
MGH9\tnote{a}& \cite{moretest}  &   3 &  15 & N  & $(-2,-2,-2)$ &  $(2,2,2)$  \\  
MGH16\tnote{a}& \cite{moretest}  &   4 &  5 & N  & $(-25,-5,-5,-1)$ & $(25,5,5,1)$  \\  
MGH26\tnote{a} & \cite{moretest} &    4 &  4 & N  & $(-1,-1,-1-1)$ &  $(1,1,1,1)$ \\ 
MGH33\tnote{a} & \cite{moretest} &    10 &  10 & Y  & $(-1,\ldots,-1)$ &  $(1,\ldots,1)$ \\ 
MHHM2&  \cite{reviewproblems}  &  2 &  3 & Y  & $(0,0)$  & $(1,1)$   \\ 
MLF1&  \cite{reviewproblems}  &    1 &  2 & N  & $0$  & $20$   \\ 
MLF2&  \cite{reviewproblems}  &    2 &  2 & N  & $(-100,-100)$ & $(100,100)$\\
MMR1  & \cite{italianos} &     2 &  2 & N  & $(0.1,0)$   & $(1,1)$   \\ 
MMR2  & \cite{italianos} &     2 &  2 & N  & $(0,0)$   & $(1,1)$   \\ 
MMR3  & \cite{italianos} &     2 &  2 & N  & $(-1,-1)$   & $(1,1)$   \\ 
MMR4  & \cite{italianos} &     3 &  2 & N  & $(0,0,0)$   & $(4,4,4)$   \\ 
MOP2 &  \cite{reviewproblems} &     2 &  2 & N  & $(-4,-4)$ &  $(4,4)$  \\ 
MOP3 & \cite{reviewproblems} &     2 &  2 & N  & $(-\pi,-\pi)$  & $(\pi,\pi)$   \\ 
MOP5&  \cite{reviewproblems} &     2 &  3 & N  & $(-30,-30)$ &  $(30,30)$ \\ 
MOP6&  \cite{reviewproblems} &     2 &  2 & N  & $(0,0)$ &  $(1,1)$ \\ 
MOP7&  \cite{reviewproblems} &     2 &  3 & Y  & $(-400,-400)$  & $(400,400)$  \\ 
PNR&  \cite{pnr} &  2 &  2 & Y & $(-2,-2)$  &   $(2,2)$\\ 
QV1&  \cite{reviewproblems}  &  10 &  2 & N & $(0.01,\ldots,0.01)$  & $(5,\ldots,5)$   \\ 
SD&  \cite{stadler1993multicriteria} &  4 &  2 & Y  & $(1,\sqrt{2},\sqrt{2},1)$  & $(3,3,3,3)$   \\ 
SK1 & \cite{reviewproblems} &      1 &  2 & N & $-100$ &$100$   \\ 
SK2 & \cite{reviewproblems} &      4 &  2 & N  & $(-10,-10,-10,-10)$ & $(10,10,10,10)$  \\ 
SLCDT1&  \cite{slcdt} &     2 &  2 & N  & $(-1.5,-1.5)$  &  $(1.5,1.5)$ \\ 
SLCDT2&  \cite{slcdt} &   10 &  3 & Y  & $(-1,\ldots,-1)$  &  $(1,\ldots,1)$ \\ 
SP1&  \cite{reviewproblems} &     2 &  2 & Y  &  $(-100,-100)$ & $(100,100)$  \\ 
SSFYY2&  \cite{reviewproblems} &     1 &  2 & N  &  $-100$ & $100$  \\ 
TKLY1&  \cite{reviewproblems} &     4 &  2 & N  &  $(0.1,0,0,0)$ & $(1,1,1,1)$  \\ 
Toi4\tnote{a}&  \cite{tointtest} &  4 &  2 & Y  & $(-2,-2,-2,-2)$  & $(5,5,5,5)$  \\ 
Toi8\tnote{a}&  \cite{tointtest} &  3 &  3 & Y  & $(-1,-1,-1,-1)$  & $(1,1,1,1)$  \\ 
Toi9\tnote{a} & \cite{tointtest}  & 4   & 4 & N & $(-1,-1,-1,-1)$  & $(1,1,1,1)$  \\ 
Toi10\tnote{a} &\cite{tointtest}  & 4 & 3 & {N}  & $(-2,-2,-2,-2)$  & $(2,2,2,2)$     \\  
VU1&  \cite{reviewproblems} &   2 &  2 & N  & $(-3,-3)$ &$(3,3)$   \\ 
VU2&  \cite{reviewproblems} &   2 &  2 & Y  & $(-3,-3)$ &$(3,3)$   \\ 
ZDT1 & \cite{doi:10.1162/106365600568202} & 30 & 2 & Y & $(0,\ldots,0)$ & $(1,\ldots,1)$ \\ 
ZDT2 & \cite{doi:10.1162/106365600568202} & 30 & 2 & N &  $(0.01,\ldots,0.01)$ & $(1,\ldots,1)$ \\ 
ZDT3 & \cite{doi:10.1162/106365600568202} & 30 & 2 & N &  $(0.01,\ldots,0.01)$ & $(1,\ldots,1)$ \\ 
ZDT4 & \cite{doi:10.1162/106365600568202} & 30 & 2 & N &  $(0.01,-5,\ldots,-5)$ & $(1,5,\ldots,5)$ \\ 
ZDT6 & \cite{doi:10.1162/106365600568202} & 10 & 2 & N &  $(0.01,\ldots,0.01)$ & $(1,\ldots,1)$ \\ 
ZLT1&  \cite{reviewproblems} &   10 &  5 & Y  & $(-1000,\ldots,-1000)$ &$(1000,\ldots,1000)$   \\ 
\hline
\end{tabular}}
 \begin{tablenotes}
      \tiny
      \item[a] This is an adaptation of a single-objective optimization problem to the multiobjective setting that can be found in \cite{ellen}.
  \end{tablenotes}
\end{threeparttable}
\caption{List of test problems.}
\label{tab:problems}
\end{table}

\vspace{12pt}
\noindent{\bf Solving the subproblems:}
We first note that  a  solution of the subproblem in \eqref{eq:opt de p}  can be calculated by solving for $\tau \in  \R$ and $u\in \R^n$ the following constrained  problem
\begin{equation} \label{vproblem}
 \begin{array}{cl}
\min_{\tau,u}   &  \tau         \\
\mbox{s.t.} & g_{j}(u)-g_{j}(x^k) +  \langle \nabla  h_j(x^k),  u-x^k \rangle \leq \tau,  \quad \forall j \in {\cal J},\\
                            & lb\preceq u\preceq ub.
\end{array}
\end{equation}
However, since $g_j(\cdot)$ in \eqref{gdef}--\eqref{Zdef} is non-differentiable, the inequalities in \eqref{vproblem}  are difficult to deal with. 
On the other hand, if we define $A_j:=[B_j; -B_j] \in\R^{2n\times n}$ and $b_j:=\delta e\in\R^{2n}$, then \eqref{gdef}--\eqref{Zdef} can be rewritten as
\begin{equation} \label{evalg}
\begin{array}{cl}
\max_{z}   &  \langle x,z \rangle         \\
\mbox{s.t.} & A_j z \preceq b_j, \\
\end{array}
\end{equation}
for which the dual problem is given by
\[\begin{array}{cl}
\min_{w}   &  \langle b_j,w \rangle         \\
\mbox{s.t.} & A_j^\top w = x, \\
                  & w \succeq 0.
\end{array}\]
By using duality theory, it follows that  \eqref{vproblem} (and thus \eqref{eq:opt de p}) is equivalent to the following linear programming problem
\begin{equation} \label{vproblem2}
 \begin{array}{cl}
\min_{\tau,u,w_j}   &  \tau         \\
\mbox{s.t.} & \langle b_j,w \rangle  -g_{j}(x^k) +  \langle \nabla  h_j(x^k),  u-x^k \rangle \leq \tau,\\
                            & A_j^\top w_j = u, \\
                            & w_j \succeq 0,\quad \forall j \in {\cal J}, \\
                            & lb\preceq u\preceq ub.
\end{array}
\end{equation}
Likewise, the subproblem \eqref{PGsub} of the Proximal Gradient method can be reformulated as the following quadratic programming problem
\begin{equation} \label{vproblem3}
 \begin{array}{cl}
\min_{\tau,u,w_j}   &  \tau+\ds\frac{\mu}{2}\|u - x^k\|^2      \\
\mbox{s.t.} & \langle b_j,w \rangle  -g_{j}(x^k) +  \langle \nabla  h_j(x^k),  u-x^k \rangle \leq \tau,\\
                            & A_j^\top w_j = u, \\
                            & w_j \succeq 0,\quad \forall j \in {\cal J}, \\
                            & lb\preceq u\preceq ub,
\end{array}
\end{equation}
for details see \cite[Section~5.2 (a)]{TanabeFukudaYamashita2019}. In our codes, we use a simplex-dual method ({\it linprog} routine) to solve \eqref{evalg} and \eqref{vproblem2}, and an interior point method  ({\it quadprog} routine) to solve \eqref{vproblem3}.

\subsection{Efficiency and robustness}

For each test problem, we considered 100 starting points randomly generated at the corresponding $\mathrm{dom}(G)=\{x\in\R^n \mid lb \preceq x\preceq ub\}$.
In this phase, each problem/starting point was considered an independent instance and solved by both algorithms.
If an approximate critical point is found, a run is considered successful regardless of the objective function value. 
Figure~\ref{fig:results} shows the results using performance profiles~\cite{dolan2002benchmarking}, comparing the algorithms with respect to: (a) CPU time; (b) number of iterations.
We emphasize that the results are similar if we consider the number of function evaluations.
In a profile performance, {\it efficiency} and {\it robustness} can be accessed on the extreme left (at 1 in the domain) and right of the graph, respectively.
As can be seen, the Conditional Gradient method was more efficient than the Proximal Gradient method considering both performance measures.
Regarding CPU time (resp. number of iterations), the efficiencies of the algorithms were $69.1\%$ and $30.3\%$ (resp. $72.8\%$ and $39.6\%$) for Algorithm~\ref{Alg:CondG} and the Proximal Gradient method, respectively.
The slightly larger difference with respect to CPU time can be explained by the fact that subproblem~\eqref{vproblem2} is simpler than subproblem~\eqref{vproblem3}, making an iteration of Algorithm~\ref{Alg:CondG} cheaper than an iteration of the Proximal Gradient method.
Both algorithms proved to be robust on the chosen set of test problems, which is in agreement with their convergence theories.
Algorithm~\ref{Alg:CondG} and the Proximal Gradient method successfully solved $98.8\%$ and $98.3\%$ of the problem instances.

\noindent\begin{figure}[H]
\centering \small
  \begin{tabular}{cc} 
  (a) CPU time &(b) Iterations\\ 
\hspace{-12pt}\includegraphics[scale=\myscale]{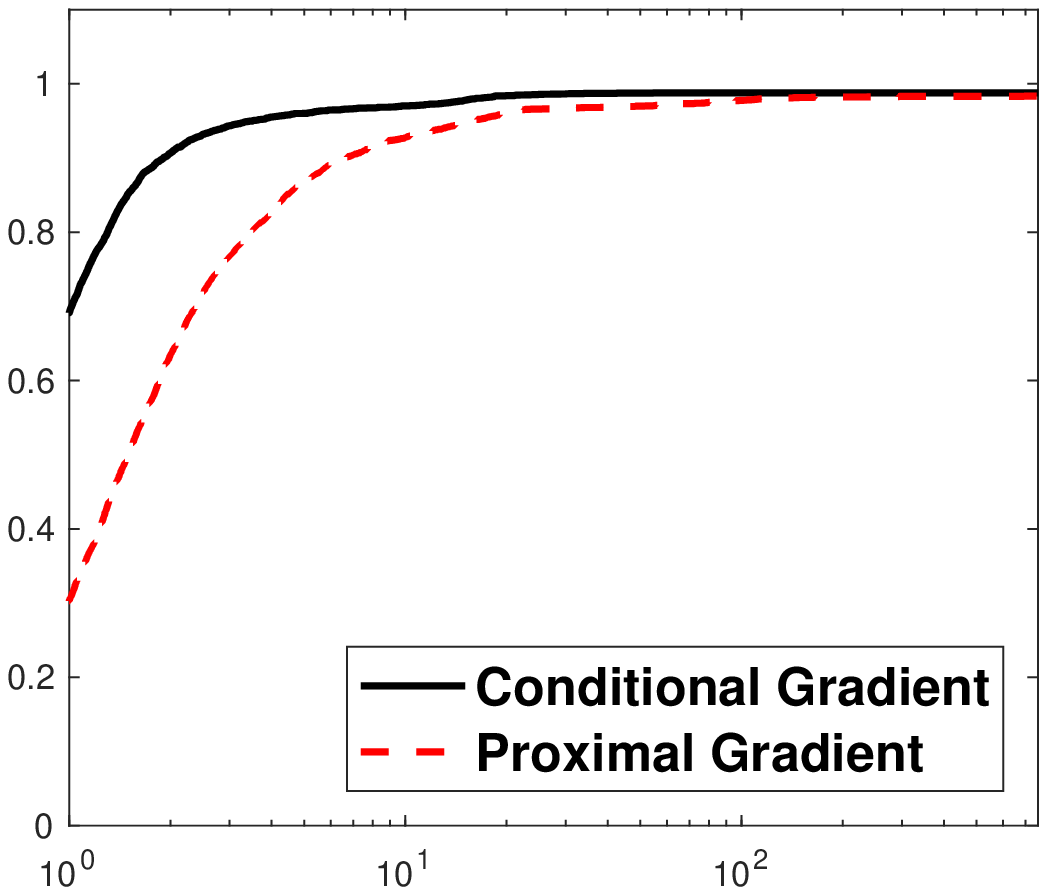}& \hspace{-12pt}\includegraphics[scale=\myscale]{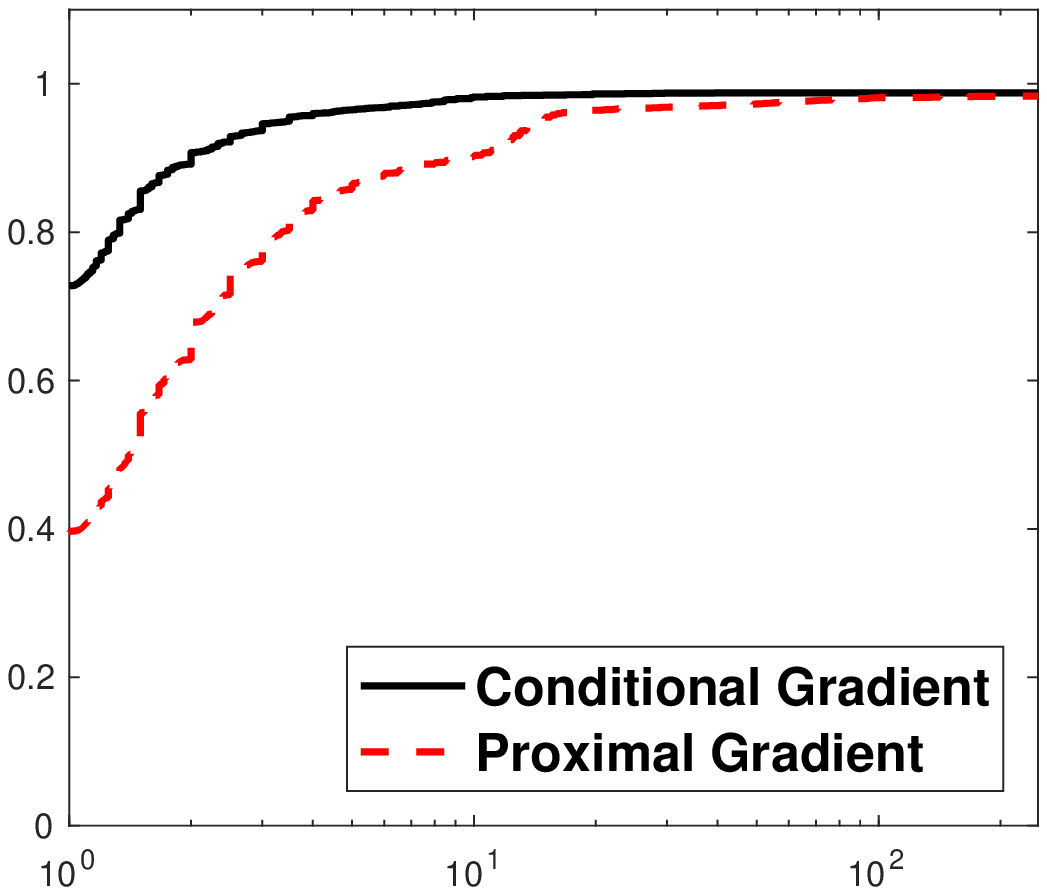} \\
\end{tabular}
\caption{Performance profiles considering 100 starting points for each test problem using as the performance measurement: (a) CPU time; (b) number of iterations.}
\label{fig:results}
\end{figure}

\subsection{Pareto frontiers}

In multiobjective optimization, we are mainly interested in estimating the Pareto frontier of a given problem.
A commonly used strategy for this task is to run an algorithm from several starting points and collect the efficient points found.
Thus, given a test problem, we run each algorithm for 2 minutes obtaining an approximation of the Pareto frontier.
We compare the results using the well-known {\it Purity} and ($\Gamma$ and $\Delta$) {\it Spread} metrics.
In summary, given a problem, the Purity metric measures the ability of an algorithm to find points on the Pareto frontier, while a Spread metric measures the ability to obtain well-distributed points along the Pareto frontier.
For a careful discussion of these metrics and their uses along with performance profiles, see \cite{doi:10.1137/10079731X}.
The results in Figure~\ref{fig:metrics} show that no significant differences are notice for the three metrics.
This suggests that the Conditional Gradient method is competitive with the Proximal Gradient method in terms of obtaining {\it good} approximations of the Pareto frontier.

\noindent\begin{figure}[H]
\centering \small
 \begin{tabular}{ccc}
 (a) Purity &(b) Spread $\Gamma$&(c) Spread $\Delta$\\
\hspace{-12pt}\includegraphics[scale=\myscale]{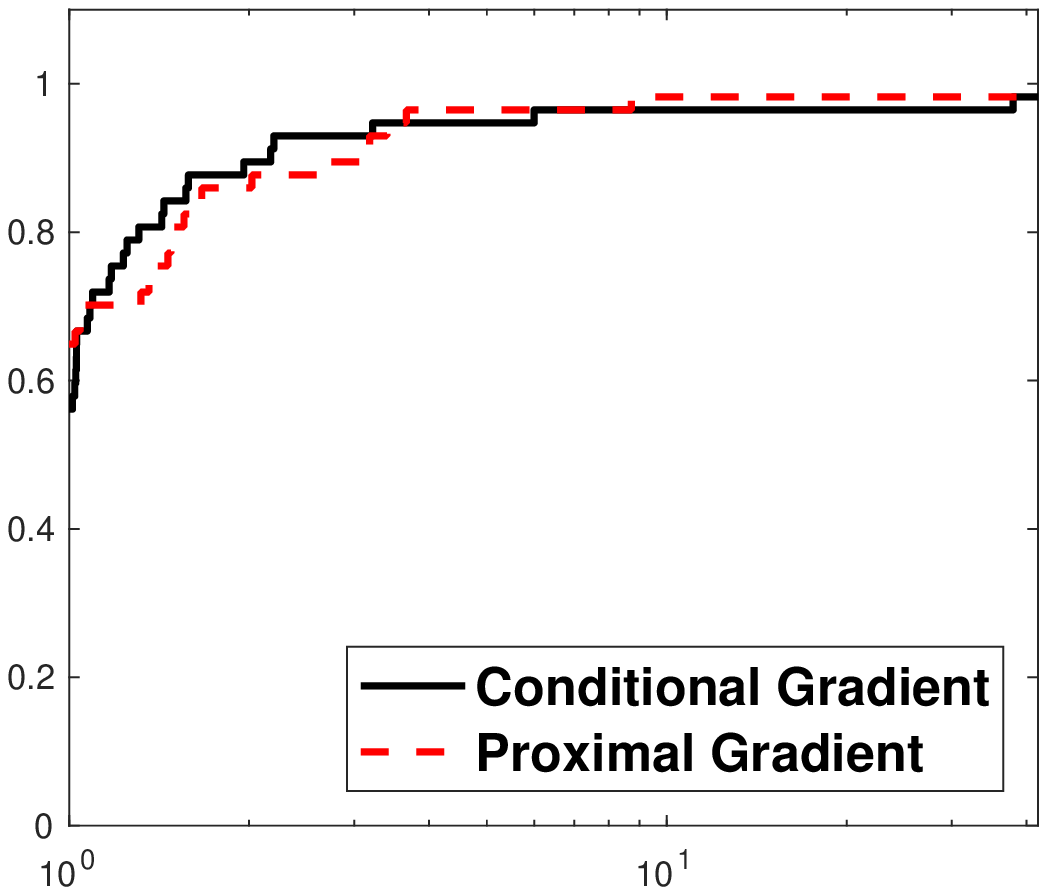}&\hspace{-12pt} \includegraphics[scale=\myscale]{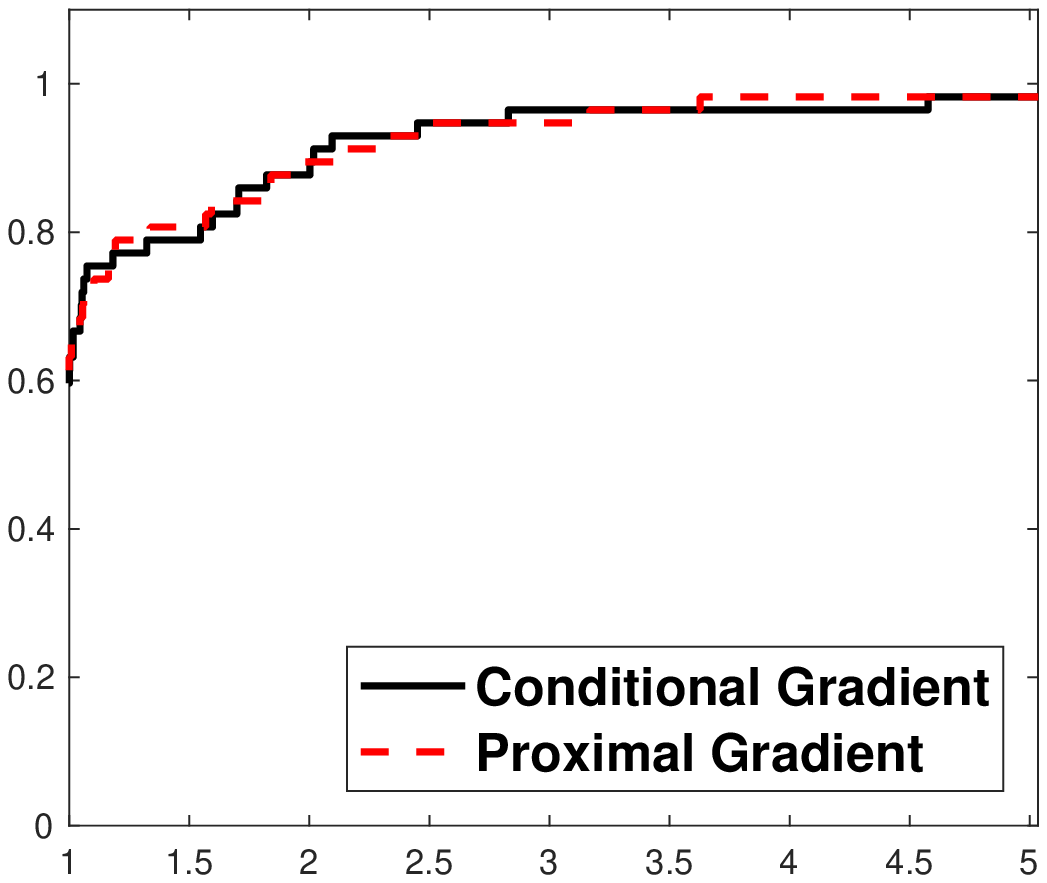}&\hspace{-12pt}\includegraphics[scale=\myscale]{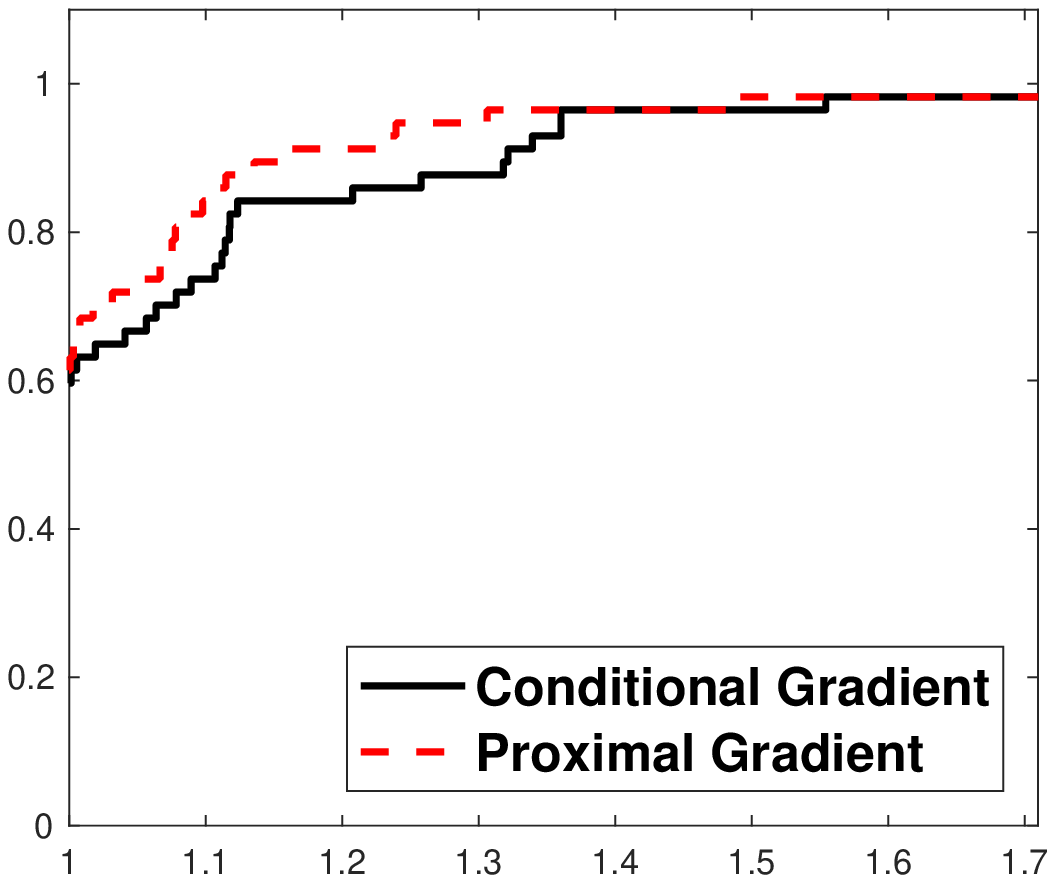}\\
\end{tabular}
\caption{Metric performance profiles considering 2 minutes for each test problem: (a) Purity; (b) Spread $\Gamma$; (c) Spread $\Delta$.}
\label{fig:metrics}
\end{figure}


We conclude the numerical experiments by illustrating the influence of the uncertainty parameter.
Figure~\ref{fig:pareto} shows the image of the Pareto critical points found by Algorithm~\ref{Alg:CondG} using 200 random starting points for problems BK1, IM1, MOP2, SD, SLCDT1, and VU2, considering the following values for  the uncertainty parameter: $\delta$ given by \eqref{deltadef} with $ \bar{\delta} = 0.02$, $0.05$, and $0.10$. As can be seen in Figure~\ref{fig:pareto}, as expected, smaller values of the uncertainty parameter are associated with better objective function values.

\noindent\begin{figure}[H]
\centering \small
 \begin{tabular}{ccc}
 (a) BK1 &(b) IM1 &(c) MOP2\\
\hspace{-12pt}\includegraphics[scale=\myscaletwo]{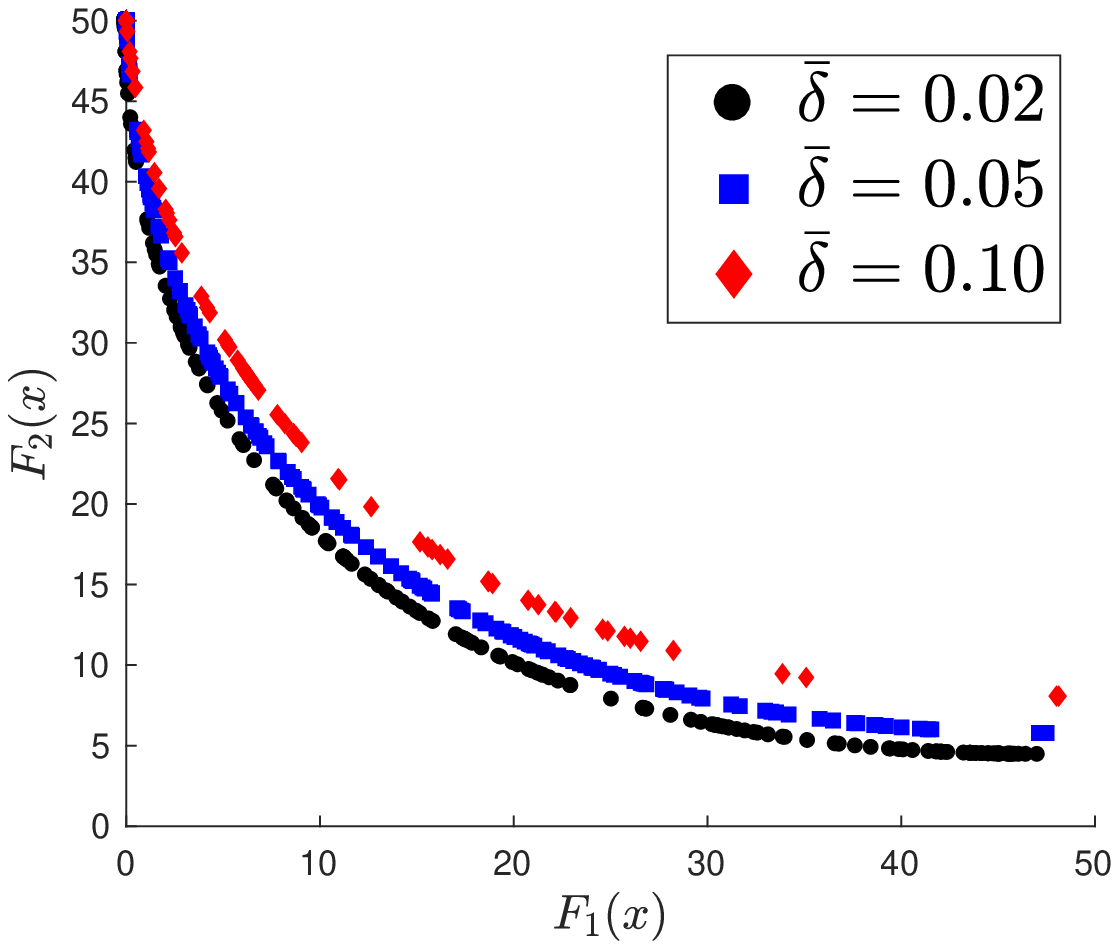}&\hspace{-12pt} \includegraphics[scale=\myscaletwo]{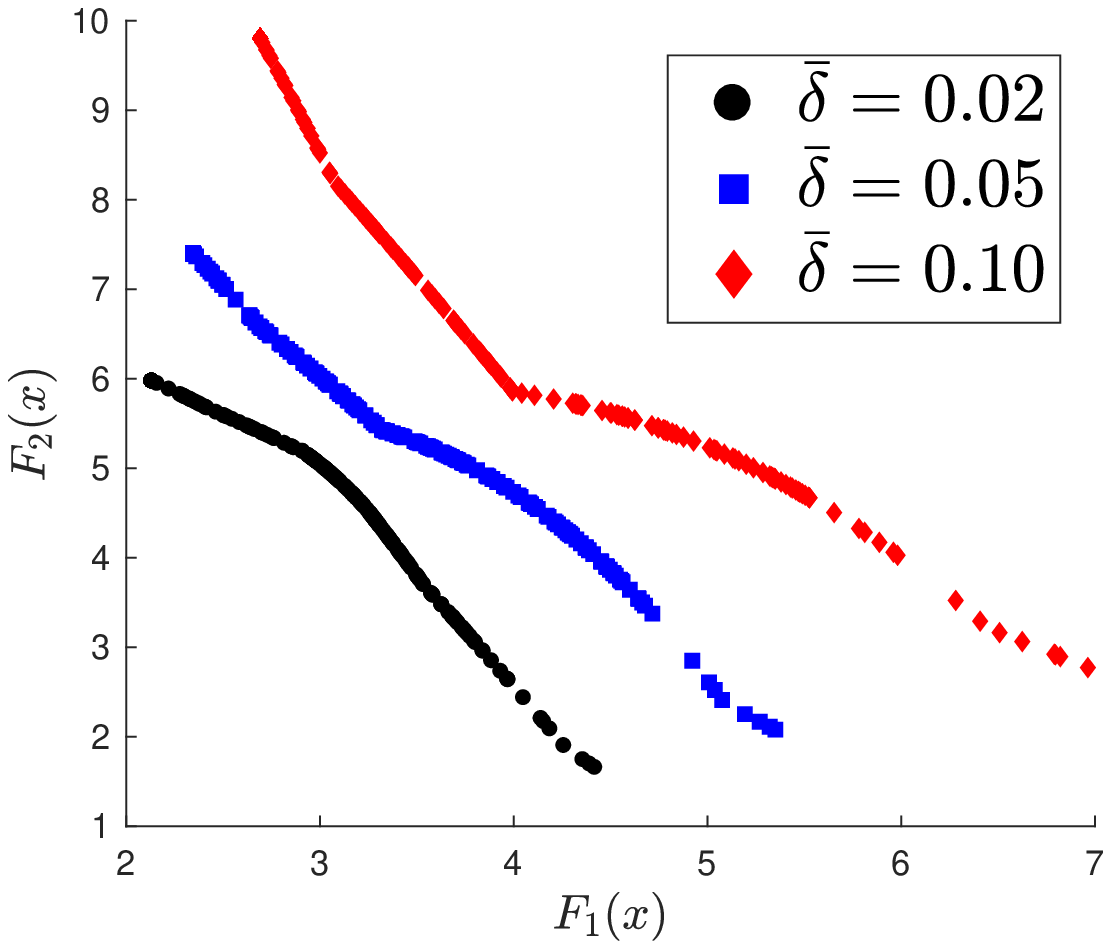}&\hspace{-12pt}\includegraphics[scale=\myscaletwo]{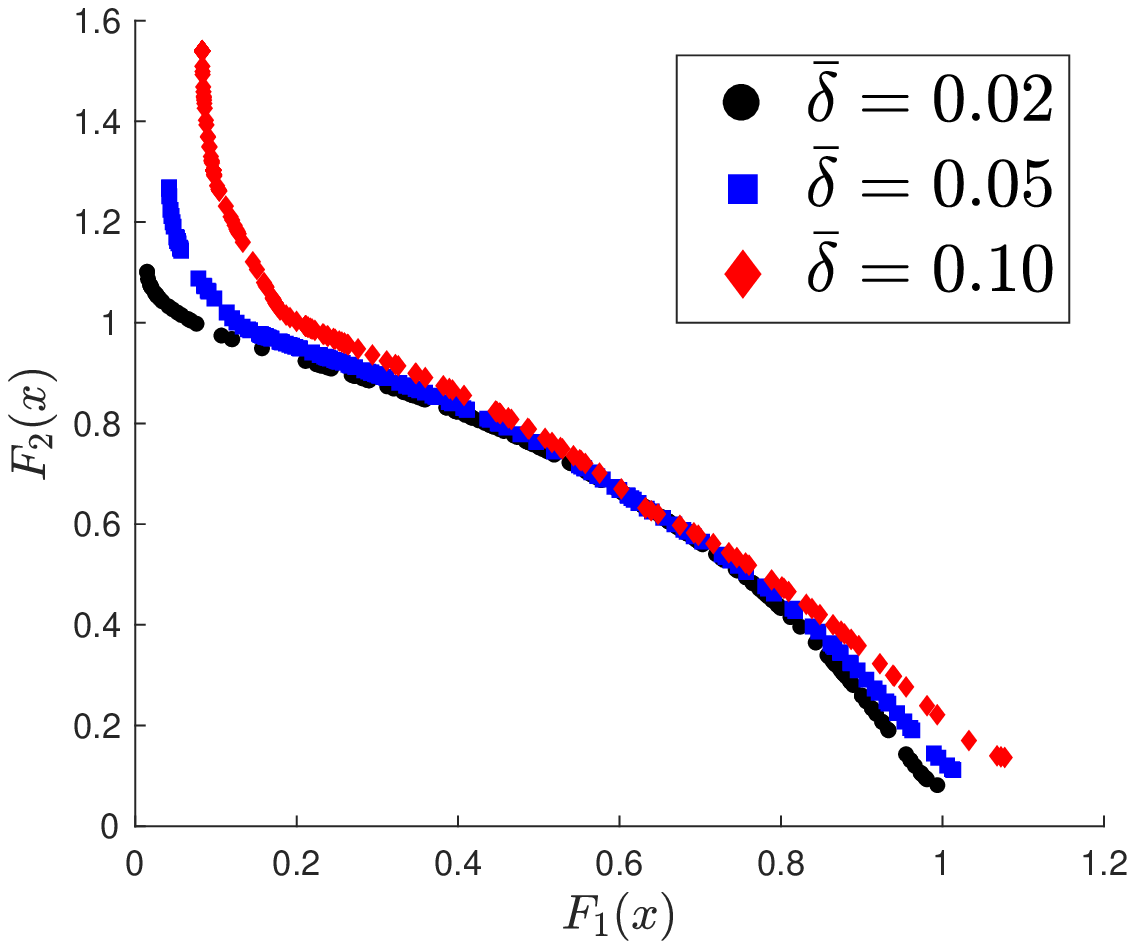}\\
 (d) SD &(e) SLCDT1 &(f) VU2 \\
\hspace{-12pt}\includegraphics[scale=\myscaletwo]{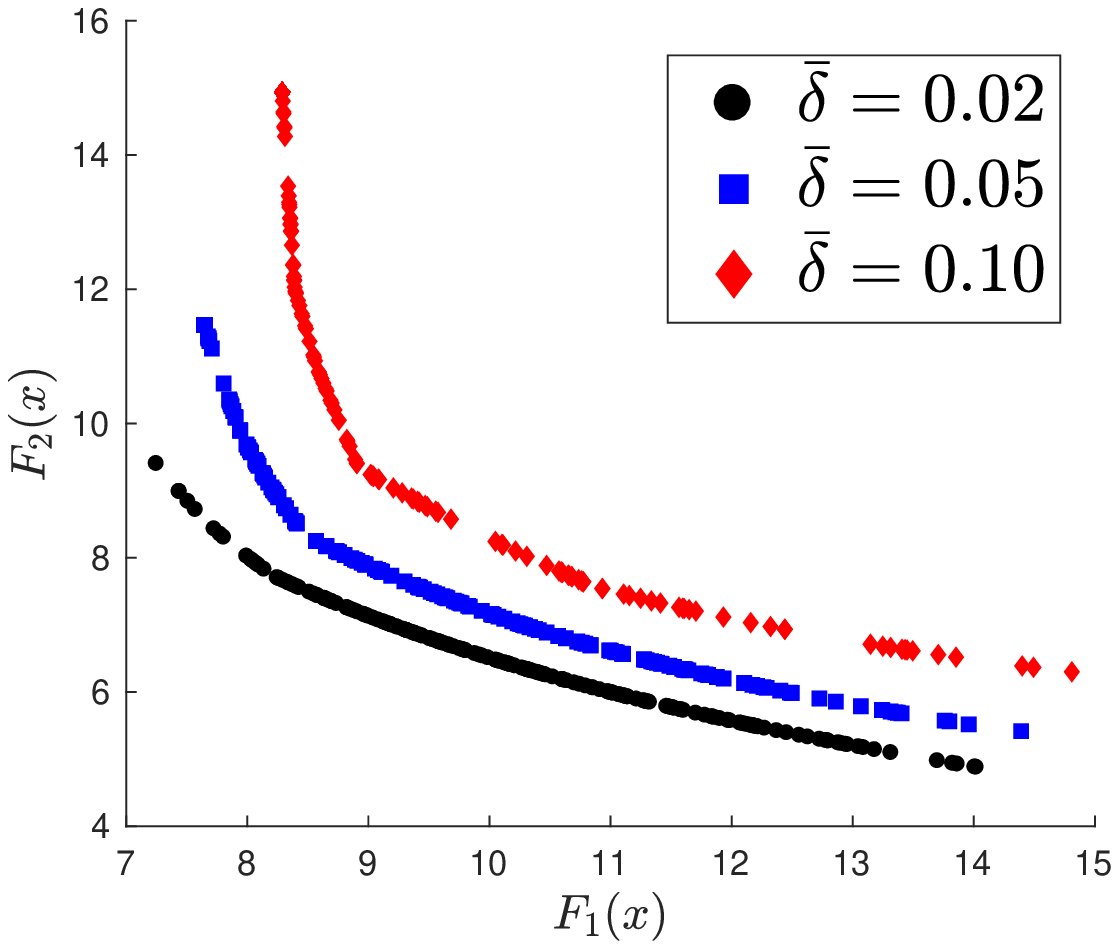}&\hspace{-12pt} \includegraphics[scale=\myscaletwo]{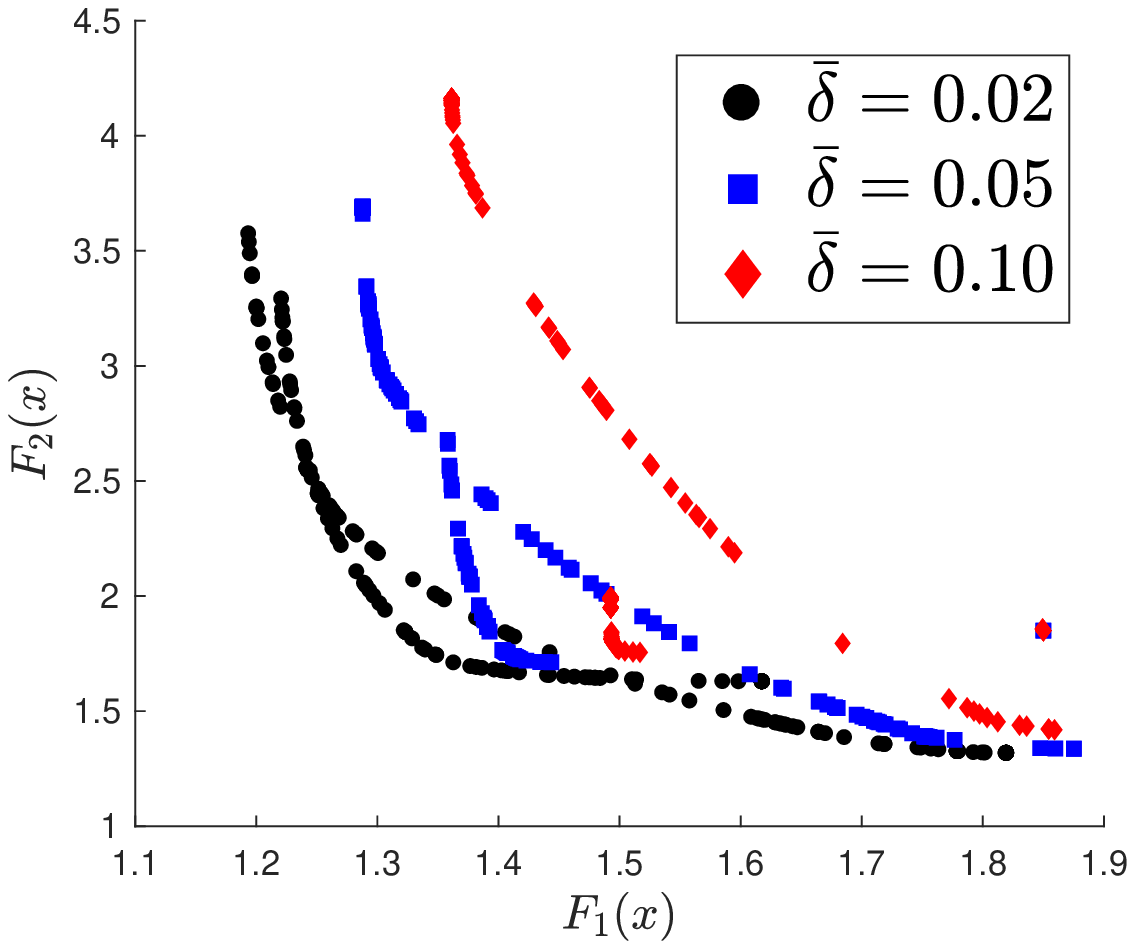}&\hspace{-12pt}\includegraphics[scale=\myscaletwo]{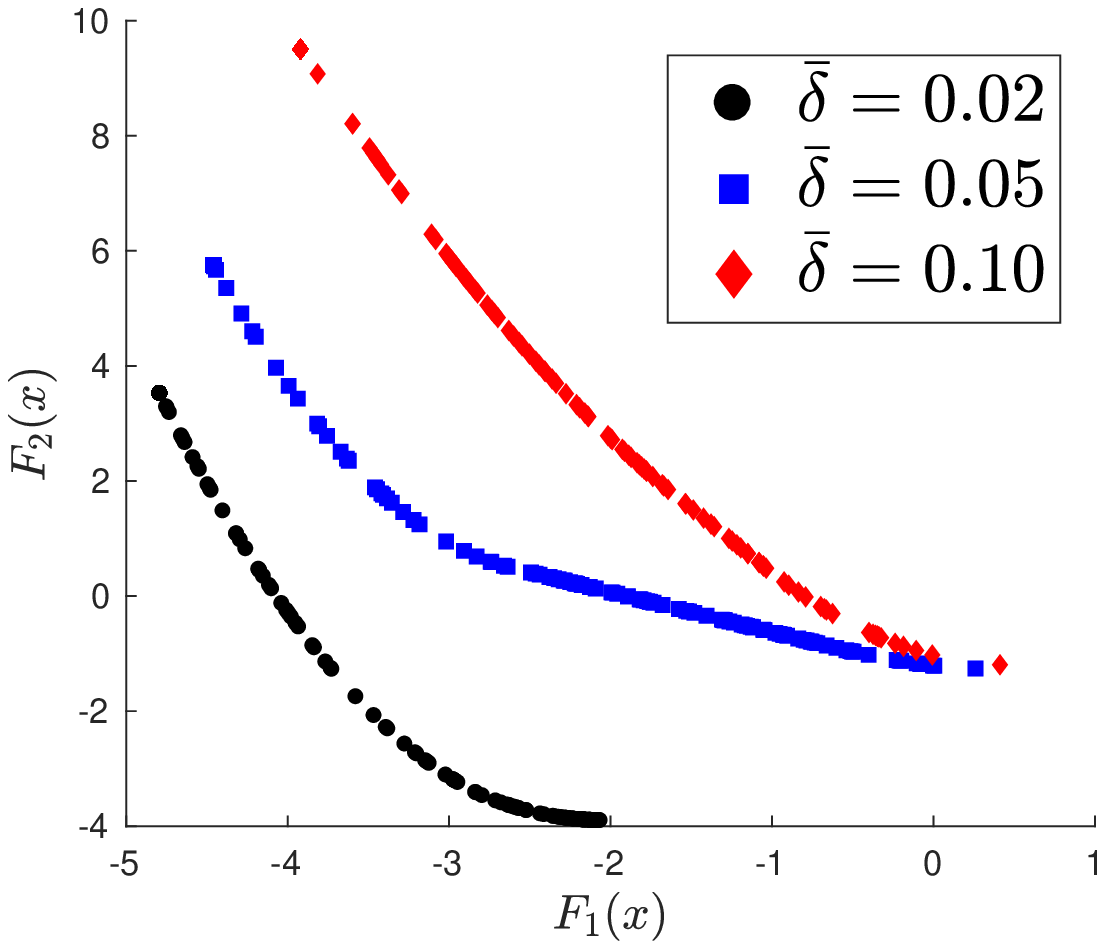}\\
\end{tabular}
\caption{Image of the Pareto critical points found by Algorithm~\ref{Alg:CondG} with different values for the uncertainty parameter for problems: (a) BK1; (b) IM1; (c) MOP2; (d) SD; (e) SLCDT1; (f) VU2.}
\label{fig:pareto}
\end{figure}

\section{Conclusions}\label{conclusions}
This paper extends the generalized conditional gradient method for multiobjective  composite optimization problems. 
Our analysis was carried out with and without convexity and Lipschitz assumptions on the smooth component of the  objective functions and considering different step size strategies.  
 The numerical results suggests that the proposed method is competitive with the Proximal Gradient method recently introduced in \cite{TanabeFukudaYamashita2019}, in terms of computational efficiency and ability to generate Pareto frontiers properly.
It would be interesting to extend the results of the present paper for composite vector optimization problem, i.e., when the partial order is induced by other underlying  cones instead of the non-negative orthant.

\section*{Data availability} 
The codes supporting the numerical experiments are freely available in the Github repository, \url{https://github.com/lfprudente/CompositeMOPCondG}.

\bibliographystyle{abbrv}
\bibliography{BibtexSumTwo}
\end{document}